\pdfoutput=1
\RequirePackage{ifpdf}
\ifpdf 
\documentclass[pdftex]{sigma}
\else
\documentclass{sigma}
\fi

\usepackage[all,2cell]{xy}
\usepackage{tikz-cd}

\numberwithin{equation}{section}

\newtheorem{Theorem}{Theorem}[section]
\newtheorem{Corollary}[Theorem]{Corollary}
\newtheorem{Lemma}[Theorem]{Lemma}
\newtheorem{Proposition}[Theorem]{Proposition}
 { \theoremstyle{definition}
\newtheorem{Definition}[Theorem]{Definition}
\newtheorem{Example}[Theorem]{Example}
\newtheorem{Remark}[Theorem]{Remark} }

\begin{document}

\allowdisplaybreaks

\newcommand{\arXivNumber}{1909.09793}

\renewcommand{\thefootnote}{}

\renewcommand{\PaperNumber}{062}

\FirstPageHeading

\ShortArticleName{Contingency Tables with Variable Margins}

\ArticleName{Contingency Tables with Variable Margins\\
(with an Appendix by Pavel Etingof)\footnote{This paper is a~contribution to the Special Issue on Algebra, Topology, and Dynamics in Interaction in honor of Dmitry Fuchs. The full collection is available at \href{https://www.emis.de/journals/SIGMA/Fuchs.html}{https://www.emis.de/journals/SIGMA/Fuchs.html}}}

\Author{Mikhail KAPRANOV~$^\dag$ and Vadim SCHECHTMAN~$^\ddag$}

\AuthorNameForHeading{M.~Kapranov and V.~Schechtman}

\Address{$^\dag$~Kavli IPMU, 5-1-5 Kashiwanoha, Kashiwa, Chiba, 277-8583, Japan}
\EmailD{\href{mailto:mikhail.kapranov@ipmu.jp}{mikhail.kapranov@ipmu.jp}}

\Address{$^\ddag$~Institut de Math\'ematiques de Toulouse, Universit\'e Paul Sabatier,\\
\hphantom{$^\ddag$}~118 route de Narbonne, 31062 Toulouse, France}
\EmailD{\href{mailto:schechtman@math.ups-tlse.fr}{schechtman@math.ups-tlse.fr}}


\ArticleDates{Received October 01, 2019, in final form June 14, 2020; Published online July 07, 2020}

\Abstract{Motivated by applications to perverse sheaves, we study combinatorics of two cell decompositions of the symmetric product of the complex line, refining the complex stratification by multiplicities. Contingency matrices, appearing in classical statistics, parametrize the cells of one such decomposition, which has the property of being quasi-regular. The other, more economical, decomposition, goes back to the work of Fox--Neuwirth and Fuchs on the cohomology of braid groups. We give a criterion for a~sheaf constructible with respect to the ``contingency decomposition'' to be constructible with respect to the complex stratification. We also study a polyhedral ball which we call the stochastihedron and whose boundary is dual to the two-sided Coxeter complex (for the root system $A_n$) introduced by T.K.~Petersen. The Appendix by P.~Etingof studies enumerative aspects of contingency matrices. In particular, it is proved that the ``meta-matrix'' formed by the numbers of contingency matrices of various sizes, is totally positive.}


\Classification{57Q05; 52B70}

\begin{flushright}
\begin{minipage}{65mm} \it To Dmitry Borisovich Fuchs\\ for his 80-th birthday, with admiration
\end{minipage}
\end{flushright}

\renewcommand{\thefootnote}{\arabic{footnote}}
\setcounter{footnote}{0}

\section{Introduction}

The $n$th symmetric product $\operatorname{Sym}^n({\mathbb C})$ can be seen as the space of monic polynomials
\[
f(x) = x^n + a_1 x^{n-1} + \cdots + a_n, \qquad a_i\in{\mathbb C}.
\]
It has a natural stratification ${\mathcal S}_{\mathbb C}$ by the multiplicities of the roots of $f$.
The topology of the stratified spaces $\big(\operatorname{Sym}^n({\mathbb C}), {\mathcal S}_{\mathbb C}\big)$ is of great importance in many areas,
 ranging from algebraic functions, braid groups, and Galois theory \cite{westerland, fox-neuwirth,fuchs}),
to representation theory and Kac--Moody algebras~\cite{BFS}. In particular, we showed in
\cite{KS-shuffle} that factorizing systems of perverse sheaves on the $\big(\operatorname{Sym}^n({\mathbb C}), {\mathcal S}_{\mathbb C}\big)$ correspond to braided Hopf algebras of a certain kind. However, despite apparent simplicity of the stratification ${\mathcal S}_{\mathbb C}$, direct study of perverse sheaves on it
is not easy and one has to ``break the symmetry'' by using various finer stratifications.

In this note we study the combinatorics of two such refinements, which are both cell decompositions.
The finest one, ${\mathcal S}^{\rm cont}$, which we call the {\em contingency cell decomposition},
has cells parametrized by contingency tables figuring in the title. It is obtained by taking into account
possible coincidences of both the real and imaginary parts of the roots. The notion of a contingency table has been introduced by the great statistician Karl Pearson in 1904, see~\cite{pearson}. The advantage of ${\mathcal S}^{\rm cont}$ is that it is a {\em quasi-regular cell decomposition} (a higher-dimensional cell can approach a lower dimensional one ``from one side only''), so a constructible sheaf on it is essentially the same as a representation of the poset of cells.

The other cell decomposition ${\mathcal S}^{\rm FNF}$, intermediate between ${\mathcal S}_{\mathbb C}$ and ${\mathcal S}^{\rm cont}$, consists of what we call {\em Fox--Neuwirth--Fuchs $($FNF$)$ cells} which generalize the cells decomposing the open stratum in~${\mathcal S}_{\mathbb C}$ (the configuration space, i.e., the classifying space of the braid group) used by Fox--Neuwirth~\cite{fox-neuwirth} and Fuchs~\cite{fuchs}. It is more economical than ${\mathcal S}^{\rm cont}$ but it is not quasi-regular. It is defined in a non-symmetric way, by looking at coincidences of the imaginary parts first and then looking at the positions of the real parts. So proceeding in the other order, we get a different cell decompostion $i{\mathcal S}^{\rm FNF}$. We prove (Theorem~\ref{prop:upper-bound}) that
\begin{gather*}
{\mathcal S}^{\rm FNF} \wedge i{\mathcal S}^{\rm FNF} = {\mathcal S}^{\rm cont}, \qquad {\mathcal S}^{\rm FNF} \vee i{\mathcal S}^{\rm FNF} = {\mathcal S}_{\mathbb C}.
\end{gather*}
The first of these equalities means that ${\mathcal S}^{\rm cont}$ is the coarsest common refinement of
${\mathcal S}^{\rm FNF}$ and~$i{\mathcal S}^{\rm FNF}$ that has connected strata. The second one means that uniting cells of ${\mathcal S}^{\rm cont}$ which lie in the same cells of ${\mathcal S}^{\rm FNF}$ and $i{\mathcal S}^{\rm FNF}$ gives the strata of ${\mathcal S}_{\mathbb C}$. In other words, it means that a~sheaf (or complex) constructible with respect to ${\mathcal S}^{\rm cont}$ is constructible w.r.t.\ ${\mathcal S}_{\mathbb C}$ if and only if it is constructible w.r.t.\ both ${\mathcal S}^{\rm FNF}$ and $i{\mathcal S}^{\rm FNF}$. This criterion will be important for our study (in progress) of perverse sheaves on $\big(\operatorname{Sym}^n({\mathbb C}), {\mathcal S}_{\mathbb C}\big)$.

Contingency tables (or {\em contingency matrices}, as we call them in the main body of the paper) give
rise to a lot of interesting combinatorics \cite{diaconis, petersen}. We study a cellular ball
called the {\em stochastihedron} ${\mathcal S}{\rm t}_n$ whose cells are labelled by contingency matrices with entries summing up to $n$. Its boundary is dual to the two-sided Coxeter complex of T.K.~Petersen~\cite{petersen} for the root system $A_n$. The stochastihedron has an interesting structure of a ``Hodge cell complex'', so that $m$-cells are subdivided into cells of type $(r,s)$, $r+s=m$ and the face inclusions are subdivided into horizintal and vertical ones, much like the de Rham differential $d$ on
a~K\"ahler manifold is decomposed into the sum of the Dolbeault differentials $\partial$ and $\overline\partial$. In a paper in preparation we use this structure for the study of perverse sheaves,
which give ``mixed sheaves'' on such complexes, that is, sheaves in the horizontal direction
and cosheaves in the vertical one.

An interesting combinatorial object is the {\em contingency metamatrix} $\mathfrak{M}(n)$. It is the $n\times n$ matrix with
\[
\mathfrak{M}(n)_{pq}= \# \bigl\{\text{contingency matrices of size $p\times q$ and sum of elements equal to $n$} \bigr\},
\]
so \looseness=-1 it describes the statistics of the ensemble of contingency matrices themselves. This matrix has a number of remarkable properties established by P.~Etingof in the appendix to this paper. Probably the most striking among them is {\em total positivity}: all minors of $\mathfrak{M}(n)$ of all sizes are positive. It seems likely that analogs of our results hold for the quotient $W\backslash {\mathbb C}^n$ for any finite real reflection group $W$. The case considered here corresponds to $W=S_n$ being the symmetric group.

\section{Contingency matrices and their contractions}

{\bf Ordered partitions.} Let ${\mathcal S}{\rm et}$ be the category of sets.
 For a set $I$, we denote by $\operatorname{Sub}(I)$ the set of subsets of $I$.
 For $m,n\in{\mathbb Z}_{\geq 0}$, $m\leq n$, we write $[m,n]=\{m, m+1,\dots, n\}$ and $[n]=[1,n]$.
 We denote ${\bf 2}^n := \operatorname{Sub}([n])$.

 An {\em ordered partition of $n$} is a sequence of positive integers summing up to $n$,{\samepage
\[
\alpha = (\alpha_1, \dots, \alpha_p)\in {\mathbb Z}_{>0}^p,\qquad \sum \alpha_i = n.
\]
The number $p$ is called the {\em length} of the ordered partition $\alpha$ and denoted $\ell(\alpha)$.}

The set of all ordered partitions of $n$ will be denoted by $\operatorname{OP}_n$ and the subset of ordered partitions of length $p$ by $\operatorname{OP}_n(p)$. We note that $\operatorname{OP}_n$ is in bijection with ${\bf 2}^{n-1}$: given $\alpha$, we write
\[
n = (1 + \dots + 1) + \dots + (1 + \dots + 1)
\]
the first parentheses contain $\alpha_1$ ones, etc. The plus signs between parentheses form a subset
\[
I = I(\alpha)\subset [n-1] = \text{the set of all pluses}.
\]
Thus
\[
|\operatorname{OP}_n(p)| ={ {n-1} \choose {p-1}}.
\]

{\bf Semisimplicial sets.}
Recall, for future reference, that an {\em augmented semisimplicial set} is a diagram
\[
Y_\bullet = \Biggl\{
\xymatrix{
Y_{-1} &\ar[l]_{\partial_0} Y_0 & \ar@<-.7ex>[l]_{\hskip .3cm \partial_0} \ar@<.7ex>[l]^{\hskip .3cm \partial_1} Y_1&
\ar@<-2.5ex>[l]_{\partial_0} \ar[l]_{\partial_1} \ar@<1.5ex>[l]^{\partial_2} Y_2 &
\ar@<-1.5ex>[l] \ar@<-0.5ex>[l] \ar@<0.5ex>[l] \ar@<1.5ex>[l] \cdots
}
\Biggr\}
\]
consisting of sets $Y_r$, $r\geq -1$ and maps $\partial_i\colon Y_r\to Y_{r-1}$, $i=0,\dots, r$, satisfying the relations
\begin{gather}\label{eq:simpl-id}
\partial_i\partial_j=\partial_{j-1}\partial_i,\qquad i<j,
\end{gather}
A {\em semisimplicial set} is a similar diagram but consisting only of $Y_r$, $r\geq 0$ (i.e., $Y_{-1}$ not present).
Elements of $Y_r$ are referred to as {\em $r$-simplices} of~$Y$. We make the following notations:
\begin{itemize}\itemsep=0pt
\item $\Delta^+_{\rm inj}$: the category of finite, possibly empty ordinals (i.e., well ordered sets) and monotone injective maps.

\item $\Delta_{\rm inj}$: the full subcategory formed by nonempty ordinals.
\end{itemize}
A semisimplicial set (resp.\ augmented semisimplicial set) is the same as a contravariant functor $Y\colon \Delta_{\rm inj}\to{\mathcal S}{\rm et}$ (resp.\ $Y\colon \Delta^+_{\rm inj}\to{\mathcal S}{\rm et}$). The set $Y_r$ is found as the value of $Y$ on the ordinal $[0,r]$ (understood as $\varnothing$ for $r=-1$). See, e.g., \cite[Section~1.2]{DK-HSS} for discussion and further references.

Returning to ordered partitions, we have the contraction maps
\begin{gather*}
\partial_i\colon \ \operatorname{OP}_n(p) \longrightarrow\operatorname{OP}_{n}(p-1), \qquad i=0, \dots, p-2,
\\
\partial_i(\alpha_1,\dots, \alpha_p) = (\alpha_1, \dots, \alpha_{i+1}+\alpha_{i+2}, \dots, \alpha_p).
\end{gather*}
These maps satisfy the simplicial identities \eqref{eq:simpl-id}
and so give an augmented semisimplicial set
\[
\operatorname{OP}_n(\bullet+2) =
 \Biggl\{
\xymatrix{
\operatorname{OP}_n(1) &\ar[l]_{\partial_0} \operatorname{OP}_n(2) & \ar@<-.7ex>[l]_{\hskip .3cm \partial_0} \ar@<.7ex>[l]^{\hskip .3cm \partial_1} \operatorname{OP}_n(3)&
\ar@<-2.5ex>[l]_{\partial_0} \ar[l]_{\partial_1} \ar@<1.5ex>[l]^{\partial_2} \operatorname{OP}_n(4) &
\ar@<-1.5ex>[l] \ar@<-0.5ex>[l] \ar@<0.5ex>[l] \ar@<1.5ex>[l] \cdots
}
\Biggr\},
\]
whose set of $r$-simplices is $\operatorname{OP}_n(r+2)$. This is nothing but the set of all geometric faces of the $(n-2)$-dimensional simplex, including the empty face.

A more standard concept is that of a {\em simplicial set}, see, e.g., \cite{gabriel-zisman, gelfand-manin}, where we have both face maps $\partial_i\colon Y_r\to Y_{r-1}$ and degeneracy maps $s_i\colon Y_r\to Y_{r+1}$. In this paper we assume familiarity with this concept. It is easy to realize $\operatorname{OP}_n(\bullet+2)$ as the set of nondegenerate simplices of an appropriate augmented simplicial set (by allowing $(\alpha_1,\dots, \alpha_p)$ with some of the intermediate~$\alpha_i$
being~$0$). The same holds for more complicated examples below, and we wil not mention it explicitly.

{\bf Contingency matrices and their bi-semisimplicial structure.} We now introduce the ``two-dimensional analog'' of the trivial considerations above.
Let us call a {\em contingency matrix} a rectangular matrix $M=\|m_{ij}\|_{i=1,\dots, p}^{j=1,\dots, q}$ of non-negative integers such that each row and each column contain at least one non-zero entry. The {\em weight} of~$M$ is defined as
\[
\Sigma M = \sum_{i,j} m_{ij} \in {\mathbb Z}_{>0}.
\]
The {\em horizontal} and {\em vertical margins} of~$M$ are ordered partitions $\sigma_{\rm hor}(M)$, $\sigma_{\rm ver}(M)$ of $n=\Sigma M$ defined by
\[
\sigma_{\rm hor}(M)_i = \sum_j m_{ij}, \qquad \sigma_{\rm ver}(M)_j=\sum_i m_{ij}.
\]
We make the following notations:
\begin{itemize}\itemsep=0pt

\item $\operatorname{CM}_n$: the set of all contingency matrices of weight $n$.

\item $\operatorname{CM}(p,q)$: the set of all contingency matrices of size $p\times q$.

\item $\operatorname{CM}_n(p,q) := \operatorname{CM}_n\cap \operatorname{CM}(p,q)$.

\item $\operatorname{CM}(\alpha, \beta)$: the set of all contingency matrices with horizontal margin~$\alpha$ and vertical margin~$\beta$. Here $\alpha, \beta\in\operatorname{OP}_n$ for some $n$.

\item $S_n$: the symmetric group of order~$n$.
\end{itemize}

\begin{Remark}\label{rem:pearson} The original setting for contingency tables given by Pearson \cite{pearson} was (in modern terminology) this. We have two random variables $x$, $y$ taking values in abstract sets $I$, $J$ of cardinalities $p$, $q$ respectively. Pearson emphasizes that in many cases fixing an embedding of $I$ or $J$ into ${\mathbb R}$ or even choosing an order on them, is unnatural. The contingency matrix $M=\|m_{ij}\|_{i\in I}^{j\in J}$ is the (un-normalized) approximation to the joint probability distribution of $x$ and $y$, taken from a sample of~$n$ trials. Thus, independence of~$x$ and $y$ means that $M$ is close to the product matrix: $m_{ij}\approx x_iy_j$. In general, various invariants of~$M$ measure deviation from independence (``contingency'').
\end{Remark}

\begin{Example} The set $\operatorname{CM}_n(n,n)$ consists of $n!$ permutation matrices
\[
M_\sigma, \sigma\in S_n, \qquad (M_\sigma)_{ij}=\begin{cases}
1, & \text{if } j=\sigma(i),
\\
0, & \text{otherwise}.
\end{cases}
\]
 \end{Example}

\looseness=-1 By a {\em bi-semisimplicial set} (resp, an {\em augmented bi-semisimplicial set} we will mean a contravariant functor $Y\colon \Delta_{\rm inj}\times\Delta_{\rm inj}\to{\mathcal S}{\rm et}$ (resp.\ $Y\colon \Delta^+_{\rm inj}\times\Delta^+_{\rm inj}\to{\mathcal S}{\rm et}$). The datum of such a functor is equivalent to the datum of the sets $Y_{r,s}$ for $r,s\geq 0$ (resp.\ $r,s\geq -1$) and two kinds of face maps: the horizontal ones $\partial'_i\colon Y_{r,a} \longrightarrow Y_{r-1,s}$, $i=0,\dots, r$, and the vertical ones $\partial''_j\colon Y_{r,s}\longrightarrow Y_{r,s-1}$, $j=0,\dots, s$, so that each group (the $\partial'_i$ as well as the $\partial''_j$) satisfies the relations~\eqref{eq:simpl-id} and the horizontal maps commute with the vertical ones. Elements of $Y_{r,a}$ are called the {\em $(r,s)$-bisimplices} of~$Y$.

Similarly to the case of simplicial sets, one has the concept of the geometric realization of a~bi-semisimplicial set, see Remarks~\ref{rem:stoch-realization} below.

In our particular case of contingency matrices, we have the horizontal and vertical contractions
\begin{gather*}
\partial_i'\colon \ \operatorname{CM}_n(p,q) \longrightarrow \operatorname{CM}_n(p-1,q), \qquad i=0,\dots, p-2,
\\
\partial_j''\colon \ \operatorname{CM}_n(p,q) \longrightarrow \operatorname{CM}_n(p, q-1), \qquad j=0,\dots, q-2,
\end{gather*}
 which add up the $(i+1)$st and the $(i+2)$nd column (resp.~$(j+1)$st and $(j+2)$nd row).
 The following is clear.

\begin{Proposition} The maps $\partial'_i$, $\partial''_j$ define an augmented bi-semisimplicial set $\operatorname{CM}_\bullet(\bullet+2,\bullet+2)$ whose $(r,s)$-bisimplices are elements of~$\operatorname{CM}_n(r+2, s+2)$.
\end{Proposition}

{\bf Contingency matrices as a (bi-)poset.} We make $\operatorname{CM}_n$ into a poset by putting $M\leq N$, if~$N$ can be obtained from $M$ by a series of contractions (of both kinds). Thus, the \mbox{$1\times 1$} mat\-rix~$(n)$ is the maximal element of $\operatorname{CM}_n$, while the minimal elements are precisely the monomial matrices~$M_\sigma$, $\sigma\in S_n$. It is convenient to arrange the poset~$\operatorname{CM}_n$ into a ``contingency square'' to indicate the order and the contractions. This square is itself an $n\times n$ ``matrix''~${\mathcal M}_n$ where, in the position~$(p,q)$, we put all the elements of the set $\operatorname{CM}_n(p,q)$.

In fact, the partial order $\leq$ can be split into two partial orders: the horizontal one $\leq'$ and
and the vertical one $\leq''$. That is, $M\leq' N$, if $N$ can be obtained from $M$ by a series
of horizontal contractions $\partial'_i$ and $M\leq '' N$, if $N$ can be obtained from $M$ by a series
of horizontal contractions $\partial''_j$. So $(\operatorname{CM}_n, \leq', \leq'')$ becomes a {\em bi-poset} (a set with two partial orders), and $\leq$ is the order generated by $(\leq', \leq'')$.

It is convenient to arrange the bi-poset $\operatorname{CM}_n$ into a ``contingency meta-square'' to indicate the orders and the contractions. This square is itself an $n\times n$ ``matrix'' ${\mathcal M}(n)$ where, in the position $(p,q)$, we put all the elements of the set $\operatorname{CM}_n(p,q)$.

\begin{Example}
The $2\times 2$ contingency meta-square ${\mathcal M}(2)$ has the form
\[
 \begin{matrix} \left(\begin{matrix} 1\\ 1\end{matrix}\right) & \longleftarrow &
\left(\begin{matrix} 1 & 0\\0 & 1\end{matrix}\right)
\left(\begin{matrix} 0 & 1\\1 & 0\end{matrix}\right) \\
\downarrow & & \downarrow\\
 (2) & \longleftarrow & \left(\begin{matrix} 1 & 1\end{matrix}\right).
\end{matrix}
\]
The arrows denote the contraction operations.
\end{Example}

\begin{Example}\label{ex:3x3maser}
The $3\times 3$ contingency meta-square ${\mathcal M}(3)$ has the form
 \[
 \xymatrix{
{ \begin{pmatrix} 1\\ 1 \\ 1\end{pmatrix} } \ar[d]
 & \ar[l]
 {\begin{matrix}
{ \begin{pmatrix} 0 & 1\\ 1 & 0 \\ 1 & 0\end{pmatrix} }&
{ \begin{pmatrix} 1 & 0\\ 0 & 1 \\ 0 & 1\end{pmatrix} }\\
{ \begin{pmatrix} 1 & 0\\ 0 & 1 \\ 1 & 0\end{pmatrix} }&
 { \begin{pmatrix} 0 & 1\\ 1 & 0 \\ 0 & 1\end{pmatrix}}\\
{ \begin{pmatrix} 1 & 0\\ 1 & 0 \\ 0 & 1\end{pmatrix} }&
 {\begin{pmatrix} 0 & 1\\ 0 & 1 \\ 1 & 0\end{pmatrix} }
 \end{matrix} }\ar[d]
 & \ar[l]
{\begin{matrix}
 {\begin{pmatrix}1 & 0 & 0\\ 0 & 1 & 0\\ 0 & 0 & 1 \end{pmatrix} }&
{ \begin{pmatrix} 0 & 1 & 0\\ 1 & 0 & 0 \\ 0 & 0 & 1\end{pmatrix}} \\
 {\begin{pmatrix} 0 & 1 & 0 \\ 0 & 0 & 1\\ 1 & 0 & 0\end{pmatrix} }&
{ \begin{pmatrix} 1 & 0 & 0 \\ 0 & 0 & 1\\ 1 & 0 & 0\end{pmatrix} }\\
{ \begin{pmatrix} 0 & 0 & 1\\ 0 & 1 & 0 \\ 1 & 0 & 0\end{pmatrix}}, &
{ \begin{pmatrix} 0 & 0 & 1 \\ 1 & 0 & 0\\ 0 & 1 & 0\end{pmatrix} }
\end{matrix} }\ar[d]
\\
{\begin{matrix} \begin{pmatrix} 2 \\ 1\end{pmatrix} \\
 \begin{pmatrix} 1 \\ 2\end{pmatrix}
\end{matrix}}\ar[d]
 &\ar[l]
{\begin{matrix}
{ \begin{pmatrix} 1 & 1\\0 & 1\end{pmatrix}}&
 \begin{pmatrix} 1 & 1\\ 1 & 0\end{pmatrix} \\
 \begin{pmatrix} 0 & 1\\ 1 & 1\end{pmatrix} &
 \begin{pmatrix} 1 & 0\\ 1 & 1\end{pmatrix} \\
 \begin{pmatrix} 2 & 0\\0 & 1\end{pmatrix} &
 \begin{pmatrix} 0 & 2\\ 1 & 0\end{pmatrix} \\
 \begin{pmatrix} 0 & 1\\ 2 & 0\end{pmatrix} &
 \begin{pmatrix} 1 & 0\\ 0 & 2\end{pmatrix}
\end{matrix}
}\ar[d]
&\ar[l]
 {\begin{matrix}
 \begin{pmatrix} 0 & 1 & 1\\ 1 & 0 & 0 \end{pmatrix} &
 \begin{pmatrix} 1 & 0 & 1 \\ 0 & 1 & 0\end{pmatrix} \\
 \begin{pmatrix} 1 & 1 & 0 \\ 0 & 0 & 1\end{pmatrix} &
 \begin{pmatrix} 1 & 0 & 0 \\ 0 & 1 & 1\end{pmatrix} \\
 \begin{pmatrix} 0 & 1 & 0 \\ 1 & 0 & 1\end{pmatrix} &
 \begin{pmatrix} 0 & 0 & 1 \\ 1 & 1 & 0\end{pmatrix}
\end{matrix}}\ar[d]
 \\
 (3)
 & \ar[l]
{ \begin{matrix}
 \begin{pmatrix} 2 & 1\end{pmatrix} &
 \begin{pmatrix} 1 & 2\end{pmatrix}
 \end{matrix}
 }
 & \ar[l]
{ \begin{matrix}
 \begin{pmatrix} 1 & 1 & 1\end{pmatrix}.
\end{matrix}}
 }
\]
 \end{Example}

{\bf Relation to the symmetric groups. Higher-dimensional analogs.} The considerations of this subsection are close to \cite[Section~6]{petersen}.

Let $\alpha= (\alpha_1,\dots,\alpha_p)\in\operatorname{OP}_n$. We have then the {\em parabolic subgroup}
in the symmetric group
\[
S_\alpha = S_{\alpha_1}\times\cdots\times S_{\alpha_p} \subset S_n.
\]
 \begin{Proposition}
 For any $\alpha,\beta\in\operatorname{OP}_n$ we have a bijection
 \[
 \operatorname{CM}_n(\alpha,\beta) \simeq S_\alpha\backslash S_n/S_\beta.
 \]
 \end{Proposition}

This is shown in \cite[Lemma~3.3]{diaconis}. For convenience of the reader we give a proof in the form that will be used later.

First of all, recall that for any group $G$ and subgroups $H,K\subset G$ we have an identification
 \begin{gather*}
 G\bigl\backslash \bigl( (G/H)\times (G/K)\bigr) \longrightarrow K\backslash G/H,
 \qquad G(g_1H, g_2K) \mapsto K g_2^{-1} g_1 H.
 \end{gather*}
 So we will construct a bijection
 \begin{gather}\label{eq:coset-prod-sym}
 S_n \bigl\backslash \bigl( (S_n/S_\alpha)\times (S_n/S_\beta)\bigr) \longrightarrow
 \operatorname{CM}_n(\alpha, \beta).
 \end{gather}

\begin{Definition}\label{def:color-part}\quad
\begin{enumerate}\itemsep=0pt
\item[(a)] A {\em colored ordered partition} of $[n]$ is a vector $A=(A_1,\dots, A_p)$ formed by nonempty subsets $A_i \subset [n]$ which make a disjoint decomposition of $[n]$. The number $p$ is called the {\em length} of~$A$ and denoted $\ell(A)$.

\item[(b)] A {\em colored contingency matrix} of weight $n$ is a matrix $K=\|K_{ij}\|$ formed by subsets $K_{ij}\subset [n]$ which make a disjoint decomposition of $[n]$ and are such that each row and each column contains at least one nonempty subset.
 \end{enumerate}
 \end{Definition}

A colored ordered partition $A$ (resp.\ colored contingency matrix $K$) gives a usual ordered partition $\alpha$ (resp.\ a usual contingency matrix $M$) with $\alpha_i=|A_i|$ (resp.~$m_{ij}=|K_{ij}|$). We denote~$\widetilde{\operatorname{CM}}_n(\alpha,\beta)$ the set of colored contingency matrices~$K$ for weight~$n$ for which the corresponding~$M$ lies in $\operatorname{CM}_n(\alpha,\beta)$. The identification~\eqref{eq:coset-prod-sym} would follow from the next claim.

 \begin{Proposition}\label{prop:CM-vs-colored}\quad
 \begin{enumerate}\itemsep=0pt
\item[$(a)$] We have an identification
 \[
 (S_n/S_\alpha) \times (S_n/S_\beta) \simeq \widetilde{\operatorname{CM}}_n(\alpha, \beta).
 \]
\item[$(b)$] We further have an identification
 \[
 \operatorname{CM}_n(\alpha,\beta) \simeq S_n\backslash \widetilde{\operatorname{CM}}_n(\alpha,\beta).
 \]
 \end{enumerate}
 \end{Proposition}

\begin{proof} (a) Note that $S_n/S_\alpha$ can be seen as the set of colored ordered partitions $(A_1,\dots, A_p)$ of $[n]$ such that $|A_i|=\alpha_i$. Similarly, if $\beta=(\beta_1,\dots, \beta_q)$, then $S_n/ S_\beta$ can be seen as the set of colored ordered partitions $(B_1,\dots, B_q)$ such that $|B_j|=\beta_j$. Now, the bijection as claimed in~(a), is obtained by sending
 \[
 \bigl( (A_1,\dots, A_p), (B_1, \dots, B_q)\bigr) \mapsto K, \qquad K_{ij} = A_i\cap B_j.
 \]

(b) This is obvious: to lift a given contingency matrix $M=\|m_{ij}\|$ to a colored one~$K$, we need to replace each entry $m_{ij}$ by a set of $m_{ij}$ elements of $[n]$, in a disjoint way. The group $S_n$ acts on the set of such lifts simply transitively.
 \end{proof}

 \begin{Remark}\label{rem:tcont-tensors}
 One can continue the pattern
 \[
 \bigl( \text{ordered partitions}, \ \text{contingency matrices}, \ \dots\bigr)
 \]
 by considering, for any $d\geq 1$, {\em $d$-valent contingency tensors} $M=\|m_{i_1,\dots, i_d}\|$
 of some format $p_1\times\cdots\times p_d$. Such an $M$ has a weight
 $n=\sum\limits_{i_1,\dots, i_p} m_{i_1,\dots, i_p}$ and $d$ margins $\sigma_\nu(M)\in\operatorname{OP}_n$, $\nu=1,\dots, d$, obtained by summation in all directions other than some given $\nu$. The set of contingency tensors with given margins $\alpha^{(1)}, \dots, \alpha^{(d)}$ is identifed with
 \[
 S_n \bigl\backslash \bigl( (S_n/S_{\alpha^{(1)}} )\times \cdots \times (S_n/S_{\alpha^{(d)}})\bigr).
 \]
As in Remark \ref{rem:pearson}, $d$-valent contingency tensors describe joint distributions of $d$-tuples
of discrete random variables. In this paper we focus on the case $d=2$ which presents special nice features
absent for $d>2$.
\end{Remark}

\section{The stochastihedron}

{\bf The stochastihedron and its properties.}
 Let $(T,\leq)$ be a poset. For $t\in T$ we denote
 \[
 T^{<t} = \{t'\in T\colon t'<t\}, \qquad T^{\leq t} = \{t'\in T\colon t'\leq t\}
 \]
 the strict and non-strict lower intervals bounded by $t$.

We also denote by $\operatorname{Nerv}_\bullet (S)$ the {\em nerve} of $T$, i.e., the simplicial set whose $r$-simplices correspond to chains $t_0\leq t_1\leq\cdots \leq t_r$ of inequalities in $T$. Nondegenerate simplices correspond to chains of strict inequalities. We denote by $\operatorname{N}(S)$ the geometric realization of the simplicial set $\operatorname{Nerv}_\bullet(T)$, i.e., the topological space obtained by gluing the above simplices together, see \cite{gabriel-zisman, gelfand-manin}. The dimension of $\operatorname{N}(T)$, if finite, is equal to the maximal length of a chain of strict inequalities. Sometimes we will, by abuse of terminology, refer to $N(T)$ as the nerve of~$T$.

We apply this to $T=(\operatorname{CM}_n,\leq)$. The space $N(\operatorname{CM}_n)$ will be called the
 {\em $n$th stochastihedron} and denote ${\mathcal S}{\rm t}_n$. We have $\dim{\mathcal S}{\rm t}_n=2n-2$.

We next show that ${\mathcal S}{\rm t}_n$ has a cellular structure of a particular kind, similar to the decomposition of a convex polytope given by its faces. Let us fix the following terminology.
\begin{itemize}\itemsep=0pt

 \item An $m$-{\em cell} is a topological space homeomorphic to the open $m$-ball
 \[ B_m^\circ =\big\{ x\in{\mathbb R}^n\colon \|x\| < 1\big\}.\]

 \item A {\em closed $m$-cell} is a topological space homeomorphic to the closed $m$-ball
 \[ B_m =\big\{ x\in{\mathbb R}^n\colon \|x\| \leq 1\big\}.\]

 \item A {\em cell decomposition} of a topological space $X$ is a filtration
 \[
 X_0\subset X_1\subset \cdots\subset X = \bigcup_{m\geq 0} X_m
 \]
by closed subspaces such that each $X_m\setminus X_{m-1}$ is a disjoint union of $m$-cells.

\item A cell decomposition is called {\em regular}, if for each cell (connected component)
$\sigma\subset X_m\setminus X_{m-1}$ the closure $\overline\sigma$ is a closed $m$-cell whose boundary
is a union of cells.

 \item A {\em $($regular$)$ cellular space} is a space with a (regular) cell decomposition.

 \item For future use, a cell decomposition of $X$ is called {\em quasi-regular}, if $X$ can be represented as $Y\setminus Z$, where $Y$ is a regular cellular space and $Z\subset Y$ a closed cellular subspace.

 \item For a quasi-regular cellular space $X$ we denote $({\mathcal C}_X,\leq)$ the poset formed by its cells with the order given by inclusion of the closures.

\end{itemize}

The following is well known, see, e.g., \cite[Remark~2.11]{hersch}.

\begin{Proposition}\label{prop:reg-cell-general} Let $X$ be a regular cellular space. Then $\operatorname{N}({\mathcal C}_X)$ is homeomorphic to $X$, being the barycentric subdivision of $X$. Further, for each $m$-cell $\sigma\in{\mathcal C}_X$ the nerve $\operatorname{N}\big({\mathcal C}_X^{\leq\sigma}\big)$ is homeomorphic to $\overline\sigma$, i.e., is a closed $m$-cell, and $N\big({\mathcal C}_X^{<\sigma}\big)$ is homeomorphic to the boundary of~$\overline\sigma$, i.e., is, topologically, $S^{m-1}$.
\end{Proposition}

We return to the poset $\operatorname{CM}_n$ and show that it can be realized as ${\mathcal C}_X$ for an appropriate regular cellular space $X$. By the above $X$ must be homeomorphic to ${\mathcal S}{\rm t}_n$, so the question is to construct an appropriate cell decomposition of ${\mathcal S}{\rm t}_n$ or, rather, to prove that certain simplicial subcomplexes in ${\mathcal S}{\rm t}_n$ are closed cells.

For any $M\in\operatorname{CM}_n$ denote by $F'(M) = \operatorname{N}\big( \operatorname{CM}_n^{<M}\big)$ and $F(M) = \operatorname{N}\big(\operatorname{CM}_n^{\leq M}\big)$. They are closed subspaces of ${\mathcal S}{\rm t}_n$. Let also $F^\circ(M) = F(M)\setminus F'(M)$, a locally closed subspace of ${\mathcal S}{\rm t}_n$. For example since the $1\times 1$ matrix $\|n\|$ is the maximal element of $\operatorname{CM}_n$, we have that $F(\|n\|)={\mathcal S}{\rm t}_n$ is the full stochastihedron, and it is the cone over $F'(\|n\|)$.

 \begin{Theorem}\label{thm:spherical}Each $F(M)$, $M\in CM_n(p,q)$, is a closed cell of dimension $2n-(p+q)$, and~$F'(M)$ is homeomorphic to the sphere $S^{2n-(p+q)-1}$. Therefore ${\mathcal S}{\rm t}_n$ has a regular cell decomposition into the cells~$F^\circ(M)$, and $\operatorname{CM}_n$ is the poset of these cells with order given by inclusion of the closures.
 \end{Theorem}

 The proof will be given in the next paragraph.

\begin{Remark}\quad
\begin{enumerate}\itemsep=0pt
\item[(a)] Theorem \ref{thm:spherical} is a consequence of Theorem~2 of~\cite{petersen} but our argument is different: it is geometric and not combinatorial.
\item[(b)] An analog of this result for contingency tensors of valency $d>2$, see Remark~\ref{rem:tcont-tensors}, does not hold.
\end{enumerate}
 \end{Remark}

{\bf The stochastihedron and the permutohedron.} 
 Here we prove Theorem~\ref{thm:spherical}. We recall that the {\em $n$th permutohedron} $P_n$ is the convex polytope in ${\mathbb R}^n$ defined as the convex hull of the~$n!$ points
 \[
 [s] = (s(1),\dots, s(n)), \qquad s\in S_n.
 \]
 By construction, the symmetric group $S_n$ acts by automorphisms of $P_n$.
 The following is well known.

\begin{Proposition}\quad
\begin{enumerate}\itemsep=0pt
\item[$(a)$] $\dim(P_n) =n-1$ and each $[s]$ is in fact a vertex of $P_n$.
\item[$(b)$] Faces of~$P_n$ are in bijection with colored ordered partitions $A=(A_1,\dots, A_p)$ of $n$, see Definition~{\rm \ref{def:color-part}(a)}. The face corresponding to $A$ is denoted $[A]$. It is the convex hull of the points $[s]$ corresponding to permutations $s$ obtained by all possible ways of ordering the elements inside each $A_i$. Thus
 \[
 [A] \simeq \prod_{i=1}^p P_{|A_i|}, \qquad \dim [A]=n-p.
 \]
\end{enumerate}
 \end{Proposition}

\begin{proof} First of all, $P_n$ is a zonotope and so the poset of its faces is anti-isomorphic to the poset of faces of ${\mathcal H}$, the associated hyperplane arrangement, see~\cite[Example~7.15 and Theorem~7.16]{ziegler}. Next, ${\mathcal H}$ is the root arrangement for the root system of type $A_{n-1}$ (cf.\ also Remark~\ref{rem;:oot-arr} below). In particular, the poset of faces of ${\mathcal H}$ is the Coxeter complex of $A_{n-1}$, which is identified with the poset of colored ordered partitions, see, e.g., \cite[pp.~40--44]{bourbaki}.
 \end{proof}

Consider now the product $P_n \times P_n$ with the diagonal action of $S_n$. Theorem \ref{thm:spherical} will follow (in virtue of Proposition~\ref{prop:reg-cell-general})
 from the next claim.

\begin{Proposition}\label{prop:quot} The quotient $S_n\backslash(P_n\times P_n)$, stratified by the images of the open faces, is a~regular cellular space with the poset of cells isomorphic to $\operatorname{CM}_n$.
 \end{Proposition}

\begin{proof} The (closed) faces of $P_n\times P_n$ are the products $[A|B]:= [A]\times [B]$
 for all pairs $A=(A_1,\dots, A_p)$, $B=(B_1,\dots, B_q)$ of colored ordered partitions of $n$. Therefore Proposition \ref{prop:CM-vs-colored}
 implies that the poset formed by the images of closed faces and their inclusions, is identified
 with $\operatorname{CM}_n$.

 It remains to show that each image of a closed face $[A|B]$ of $P_n\times P_n$ is a closed cell,
 i.e., is homeomorphic to a closed ball. This image is the quotient
 $S_{[A|B]}\backslash [A|B]$, where $S_{[A|B]}\subset S_m$ is the subgroup preserving
 $[A|B]$ as a set. For a set $I$ let us denote $S_I=\operatorname{Aut}(I)$ the symmetric group of automorpisms of $I$.
 Then it is immediate that
 \[
 S_{[A|B]} = \prod_{i=1}^p \prod_{j=1}^q S_{A_i\cap B_j} \subset S_n.
 \]
 Further, let us denote, for a finite set $I$
 \[
 {\mathbb R}^I_0 = \left\{ (x_ k)_{k\in I} \in{\mathbb R}^I\colon \sum x_k=0\right\}
 \]
and write ${\mathbb R}^n_0={\mathbb R}^{[n]}_0$. Thus, $P_n$ is parallel to ${\mathbb R}^n_0$ and
each face $[A]$ is parallel to $\prod {\mathbb R}^{A_i}_0$.

Let $[A|B]^\circ$ be the interior of the face $[A|B]$. By the above, it is a translation of an open set in
\[
\bigg(\prod_i {\mathbb R}^{A_i}_0\bigg) \times \bigg(\prod_j {\mathbb R}^{B_j}_0\bigg) \subset {\mathbb R}^n\times {\mathbb R}^n,
\]
a translation, moreover, by a vector invariant with respect to $S_{[A|B]}$.
So to prove that each $[A|B]$ is a closed cell (and each $[A|B]^\circ$ is an open cell),
it suffices to establish the following.

\begin{Lemma}\label{lem:quot}
For each $A$, $B$ as above, the quotient
\[
\prod_{i,j} S_{A_i\cap B_j}\biggl\backslash
 \biggl(\bigg(\prod_i {\mathbb R}^{A_i}_0\bigg) \times \bigg(\prod_j {\mathbb R}^{B_j}_0\bigg)\biggr)
\]
is homeomorphic to the Euclidean space $($of dimension $2n-p-q)$.
\end{Lemma}

\begin{proof}[Proof of the lemma] Denote the quotient in question by $Q$.
Consider first the bigger space
\[
 Q' = \prod_{i,j} S_{A_i\cap B_j}\biggl\backslash
 \biggl(\bigg(\prod_i {\mathbb R}^{A_i}\bigg) \times \bigg(\prod_j {\mathbb R}^{B_j}\bigg)\biggr)
\]
which contains $Q$ as a closed subset. We note that
\[
 Q' = \prod_{i,j} S_{A_i\cap B_j}\bigl\backslash \big({\mathbb R}^n\times{\mathbb R}^n\big) = \prod_{i,j} S_{A_i\cap B_j} \bigl\backslash \big({\mathbb R}^2\big)^n = \prod_{i,j} \big( S_{A_i\cap B_j}\backslash \big({\mathbb R}^2\big)^{A_i\cap B_j}\big).
\]
Now, for any finite set $I$, the quotient
\[
S_I\backslash \big({\mathbb R}^2\big)^I = S_I\backslash {\mathbb C}^I = \operatorname{Sym}^{|I|}({\mathbb C}) \simeq {\mathbb C}^{|I|}
\]
is the $|I|$th symmetric product of ${\mathbb C}$ and so is identified (as an algebraic variety and hence
as a topological space) with ${\mathbb C}^{|I|}$. The coordinates in this new ${\mathbb C}^{|I|}$ are the elementary symmetric functions of the coordinates $x_k$, $i\in I$, in the original ${\mathbb C}^I$. In particular, one of these coordinates is $\sigma_{1,I} = \sum\limits_{k\in I} x_k$, the sum of the original coordinates.

Applying this remark to $I=A_i\cap B_j$ for all $i$, $j$, we see, first of all, that
\[
Q' \simeq \prod_{i,j} {\mathbb C}^{|A_i\cap B_j|} \simeq {\mathbb C}^n.
\]
Second, to identify $Q$ inside $Q'$, we need to express the effect, on the quotient, of replacing each~${\mathbb R}^{A_i}$ by ${\mathbb R}^{A_i}_0$ and each~${\mathbb R}^{B_j}$ by ${\mathbb R}^{B_j}_0$, i.e., of imposing the zero-sum conditions throughout. Let us view the first ${\mathbb R}^n=\prod _i {\mathbb R}^{A_i}$ as the real part and the second ${\mathbb R}^n=\prod_j {\mathbb R}^{B_j}$ as the imaginary part of~${\mathbb C}^n$. Then the zero-sum condition on an element of ${\mathbb R}^{A_i}$ is expressed by vanishing of $\sum_j \sigma_{1, A_i\cap B_j}$ applied to the real part of a point of $\prod_{i,j} \operatorname{Sym}^{|A_i\cap B_j|}({\mathbb C})$. Similarly, the zero sum condition on an element of ${\mathbb R}^{B_j}$ is expressed by vanishing of $\sum_i\sigma_{1, A_i\cap B_j}$ applied to the imaginary part a point of $\prod_{i,j} \operatorname{Sym}^{|A_i\cap B_j|}({\mathbb C})$. So $Q$ is specified, inside $Q'\simeq {\mathbb C}^n$, by vanishing of a collection of ${\mathbb R}$-linear functions and so is homeomorphic to a real Euclidean space as claimed.
\end{proof}

 Lemma \ref{lem:quot} being proved, this finishes the proof of Proposition \ref{prop:quot} and Theorem
\ref{thm:spherical}.
\end{proof}

{\bf Examples and pictures.}
 We illustrate the above concepts in low dimensions.

\begin{Example}
 The $2$nd stochastihedron ${\mathcal S}{\rm t}_2$ is a bigon:
 
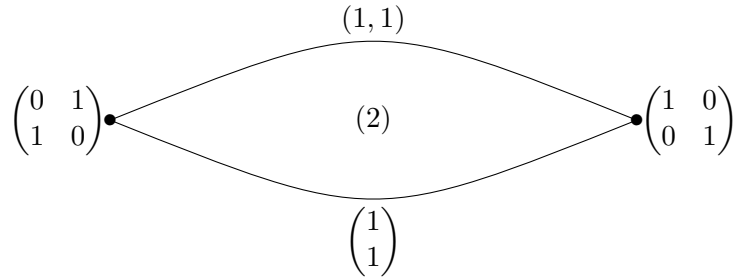
\begin{figure}[h!]
 \centering
 \begin{tikzpicture}[scale=0.7]

 \node at (5,0){$\bullet$};
 \node at (-5,0){$\bullet$};
 \draw (-5,0) .. controls (0,2) .. (5,0);
 \draw (-5,0) .. controls (0,-2) .. (5,0);
\node at (0,0){$(2)$};
\node at (6,0){$\begin{pmatrix} 1&0\\0&1
\end{pmatrix}$};
\node at (-6,0){$\begin{pmatrix} 0&1\\1&0
\end{pmatrix}$};
\node at (0,1.9){$(1,1)$};
\node at (0,-2.3){$\begin{pmatrix}1\\1
\end{pmatrix} $};
 \end{tikzpicture}
 \end{figure}
 \end{Example}

\begin{Example} \label{ex:St-3} The $3$rd stochastihedron ${\mathcal S}{\rm t}_3$ is a $4$-dimensional cellular complex with $33$ cells, corresponding to the matrices in the ``contingency square'' ${\mathcal A}_3$ of Example~\ref{ex:3x3maser}:
\begin{itemize}\itemsep=0pt
\item $6$ vertices; they correspond to $3\times 3$ permutation matrices in the upper right corner;
\item $12$ edges; they correspond to $2\times 3$ and $3\times 2$ matrices;
\item $10$ $2$-faces, more precisely:
\begin{itemize}\itemsep=0pt
\item $4$ bigons corresponding to $2\times 2$ matrices $M$ which contain an entry $2$;
\item $4$ squares corresponding to $2\times 2$ matrices $M$ which cosists of $0$'s and $1$'s only;
\item $2$ hexagons $P_3$, corresponding to $1\times 3$ and $3\times 1$ matrices;
\end{itemize}
\item $4$ $3$-faces, of the shape we call {\it hangars}, see Fig.~\ref{fig:hangar} below. They correspond to $2\times 1$ and $1\times 2$ matrices;
\item one $4$-cell corresponding to the matrix $(3)$.
\end{itemize}

\end{Example}

 \begin{Remark}
 Note that the boundaries of the cells of ${\mathcal S}{\rm t}_n$ come from {\em decontractions} (acting to the right and upwards
 in the contingency meta-square ${\mathcal M}(3)$, in the above example) and not contractions. Therefore ${\mathcal S}{\rm t}_n$ is not the realization of the
 bi-semisimplicial set $\operatorname{CM}_n(\bullet+2,\bullet +2)$
 but, rather, the Poincar\'e dual cell complex to it. Because of this, Theorem \ref{thm:spherical} is non-trivial.
 For the nature of the realization itself (which is a cellular space by its very construction), see Remark
 \ref {rem:stoch-realization}(a) below.
 \end{Remark}

\begin{Example} Here we describe one hangar corresponding to the matrix $(2,1)^t= \begin{pmatrix} 2\\1\end{pmatrix}$ (the other hangars look similarly). This particular hangar is a cellular $3$-ball, whose cells correspond to elements of the lower interval $\operatorname{CM}_3^{\leq (2,1)^t}$, which has the form
 
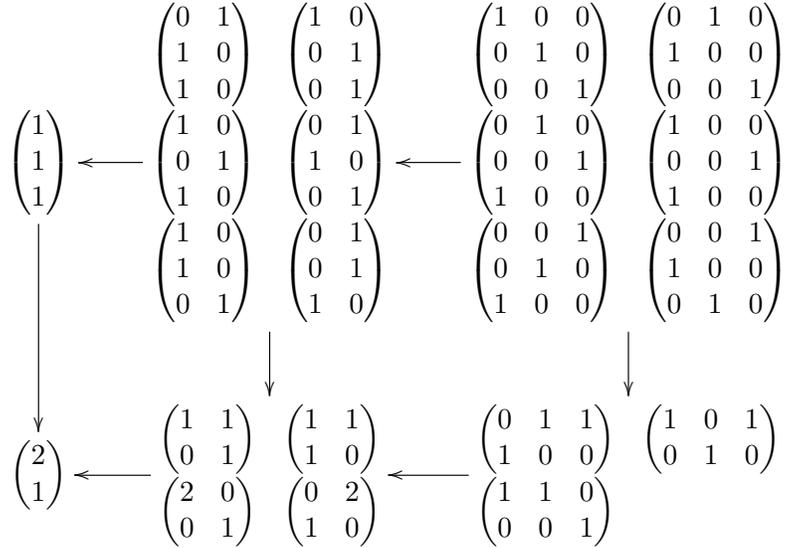
\[
 \xymatrix{
 {\begin{pmatrix} 1\\1\\1\end{pmatrix}} \ar[d] & \ar[l]
 {\begin{matrix}
 {\begin{pmatrix} 0 & 1\\ 1 & 0 \\ 1 & 0 \end{pmatrix}} & {\begin{pmatrix} 1 & 0\\ 0 & 1 \\ 0 & 1 \end{pmatrix}} \\
 {\begin{pmatrix} 1 & 0\\ 0 & 1 \\ 1 & 0 \end{pmatrix}} & {\begin{pmatrix} 0 & 1\\ 1 & 0 \\ 0 & 1 \end{pmatrix}} \\
 {\begin{pmatrix} 1 & 0\\ 1 & 0 \\ 0 & 1 \end{pmatrix}} & {\begin{pmatrix} 0 & 1\\ 0 & 1 \\ 1 & 0 \end{pmatrix}}
 \end{matrix}} \ar[d] & \ar[l] \ar[d]
 {\begin{matrix}
 {\begin{pmatrix} 1 & 0 & 0\\ 0 & 1 & 0\\ 0 & 0 & 1 \end{pmatrix}} & {\begin{pmatrix} 0 & 1 & 0\\ 1 & 0 & 0 \\ 0 & 0 & 1 \end{pmatrix}} \\
 {\begin{pmatrix} 0 & 1 & 0 \\ 0 & 0 & 1\\ 1 & 0 & 0 \end{pmatrix}} & {\begin{pmatrix} 1 & 0 & 0 \\ 0 & 0 & 1\\ 1 & 0 & 0 \end{pmatrix}} \\
 {\begin{pmatrix} 0 & 0 & 1\\ 0 & 1 & 0 \\ 1 & 0 & 0 \end{pmatrix}} & {\begin{pmatrix} 0 & 0 & 1 \\ 1 & 0 & 0\\ 0 & 1 & 0 \end{pmatrix}}
 \end{matrix}}
 \\
 {\begin{pmatrix} 2\\1\end{pmatrix}}& \ar[l]
 {\begin{matrix}
 {\begin{pmatrix} 1 & 1\\0 & 1 \end{pmatrix}} & {\begin{pmatrix} 1 & 1\\ 1 & 0 \end{pmatrix}} \\
 {\begin{pmatrix} 2 & 0\\0 & 1 \end{pmatrix}} & {\begin{pmatrix} 0 & 2\\ 1 & 0 \end{pmatrix}}
 \end{matrix}}
 & \ar[l]
 {\begin{matrix}
 {\begin{pmatrix} 0 & 1 & 1\\ 1 & 0 & 0 \end{pmatrix}} & {\begin{pmatrix} 1 & 0 & 1 \\ 0 & 1 & 0 \end{pmatrix}} \\
 {\begin{pmatrix} 1 & 1 & 0 \\ 0 & 0 & 1 \end{pmatrix}}
 \end{matrix}}
 }
 \]
 The boundary ($2$-dimensional) cells are as follows:
 \begin{itemize}\itemsep=0pt
 \item one hexagon corresponding to $\left(\begin{matrix} 1\\ 1 \\ 1\end{matrix}\right)$;
\item $2$ squares corresponding to $\left(\begin{matrix} 1 & 1\\0 & 1\end{matrix}\right)$,
$\left(\begin{matrix} 1 & 1\\ 1 & 0\end{matrix}\right)$;

\item and $2$ bigons corresponding to $\left(\begin{matrix} 2 & 0\\0 & 1\end{matrix}\right)$,
$\left(\begin{matrix} 0 & 2\\ 1 & 0\end{matrix}\right)$.
 \end{itemize}
 See Fig.~\ref{fig:hangar}, where the two thin straight lines in the middle represent the highest visible points of the upper surface of the hangar and should not be confused with $1$-dimensional cells depicted as
 thicker lines.
 \end{Example}

 \begin{figure}[h!] \centering
 \begin{tikzpicture}[scale=0.3]
\node (a) at (10, -3){}; \fill(a) circle (0.2);
\node (b) at (15,2){}; \fill(b) circle (0.2);
\node(c) at (8,4){}; \fill(c) circle (0.2);
\node (d) at (-9,2){}; \fill(d) circle (0.2);
\node (e) at (-14, -3){}; \fill(e) circle (0.2);
\node (f) at (1,-6){}; \fill(f) circle (0.2);

\draw[line width =1.2] (e.center) -- (f.center) -- (a.center) -- (b.center);
\draw [dotted, line width=1.9] (e.center) -- (d.center) -- (c.center) -- (b.center);

\draw [line width=1.2] (10,-3) .. controls (11,2) and (12,3) .. (15,2);
\draw [line width=1.2] (-14,-3) .. controls (-13,2) and (-12,3) .. (-9,2);
\draw [line width=1.2] (1,-6) .. controls (3, 5) and (4, 6) .. (8,4);

\draw (-11.45,2.28) -- (5,4.78);

\draw (5,4.78) -- (12,2);

 \end{tikzpicture}
 \caption{A hangar having 2 bigons, 2 (curved) squares and a hexagon in the boundary.} \label{fig:hangar}
 \end{figure}
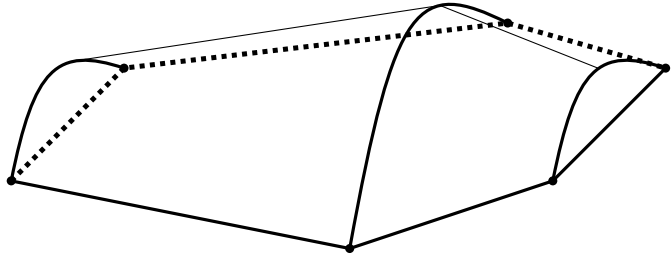

\section{The stochastihedron and symmetric products}

{\bf The symmetric product and its complex stratification ${\mathcal S}_{\mathbb C}$.} Let ${\mathcal P}_n$ be the set of (unordered) partitions $\alpha = (\alpha_1\geq\cdots\geq\alpha_p)$, $\sum\alpha_i=n$ of $n$. For any ordered partition $\beta\in\operatorname{OP}_n$ let $\overline\beta\in{\mathcal P}_n$ be the corresponding unordered partition (we put the parts of $\beta$ in the non-increasing order).

We consider the symmetric product $\operatorname{Sym}^n({\mathbb C}) = S_n\backslash {\mathbb C}^n$ with the natural projection
 \begin{gather}\label{eq:pi}
 \pi\colon \ {\mathbb C}^n\longrightarrow \operatorname{Sym}^n({\mathbb C}).
 \end{gather}
It is classical that $\operatorname{Sym}^n({\mathbb C})\simeq{\mathbb C}^n$, the isomorphism given by the elementary symmetric functions. We can view points ${\bf z}$ of $\operatorname{Sym}^n({\mathbb C})$ in either of two ways:
\begin{itemize}\itemsep=0pt
\item As effective divisors ${\bf z} = \sum\limits_{z\in{\mathbb C}} \alpha_z\cdot z$ with $\alpha_z\in{\mathbb Z}_{\geq 0}$, of degree $n$, that is, $\sum_z \alpha_z=n$.

\item As unordered collections ${\bf z}=\{z_1,\dots, z_n\}$ of $n$ points in ${\mathbb C}$, possibly with repetitions.
\end{itemize}

Viewing ${\bf z}$ as a divisor, we have an ordered partition $\operatorname{Mult}({\bf z}) = (\alpha_1\geq\cdots\geq\alpha_p)$,
 called the {\em multiplicity partition} of ${\bf z}$, which is obtained by arranging the $\alpha_z$ in a non-increasing way. For a given $\alpha\in{\mathcal P}_n$ the {\em complex stratum} $X_\alpha^{\mathbb C}$ is formed by all ${\bf z}$ with $\operatorname{Mult}({\bf z})=\alpha$.
 These strata are smooth complex varieties forming the {\em complex stratification} ${\mathcal S}_{\mathbb C}$ of $\operatorname{Sym}^n({\mathbb C})$.

Our eventual interest is in constructible sheaves and perverse sheaves on $\operatorname{Sym}^n({\mathbb C})$ which are smooth with respect to the stratification ${\mathcal S}_{\mathbb C}$. We now review various refinements of the stratification ${\mathcal S}_{\mathbb C}$ obtained by taking into account the real and imaginary parts of the points $z_\nu\in{\mathbb C}$ forming a point ${\bf z}\in\operatorname{Sym}^n({\mathbb C})$.

{\bf The contingency cell decomposition.} Let ${\bf z}=\{z_1,\dots, z_n\}\in\operatorname{Sym}^n({\mathbb C})$. Among the numbers $\operatorname{Re}(z_1),\dots, \operatorname{Re}(z_n)$ there may be some coincidences. Let $x_1<\cdots < x_p$ be all the values of $\operatorname{Re}(z_\nu)$ in the increasing order (ignoring possible repetitions). Similarly for the imaginary parts: let $y_1<\cdots < y_q$ be all the values of $\operatorname{Im}(z_\nu)$ in the increasing order, see Fig.~\ref{fig:S-a-Im}. We get a~contingency matrix
 \[
 \mu({\bf z}) = \|\mu_{ij}({\bf z})\|_{i=1,\dots, p}^{j=1,\dots, q} \in \operatorname{CM}_n(p,q),\qquad
 \mu_{ij}({\bf z}) = |\{\nu\colon \operatorname{Re}(z_\nu)=x_i \text{ and } \operatorname{Im}(z_\nu)=y_q\}|.
 \]

\begin{figure}[h!] \centering
 \begin{tikzpicture}[scale=.4, baseline=(current bounding box.center)]

 \draw[->] (-2,0) -- (13,0);

 \draw[->] (0,-2) -- (0,11);

 \draw[dashed] (0,2) -- (11,2);

 \draw[dashed] (0,4) -- (11,4);

 \draw[dashed] (0,7) -- (11,7);

 \node at (-1,2){$y_1$}; \node at (-1,4){$y_2$};
 \node at (-1,5.5){$\vdots$};
 \node at (-1,7){$y_q$};

 \node at (2,7){$\bullet$};
 \node at (2,2){$\bullet$};
 \node at (8,2){$\bullet$};
 \node at (5,4){$\bullet$};
 \node at (8,7){$\bullet$}; \node at (8,4){$\bullet$};

 \draw[dashed] (2,0) -- (2,9);
 \draw[dashed] (5,0) -- (5,9);
 \draw[dashed] (8,0) -- (8,9);

 \node at (2,-1){$x_1$}; \node at (5,-1){$x_2$}; \node at (8,-1){$x_p$};
 \node at (6.5, -1){$\cdots$};

 \node at (2.9, 2.5){ \small{$\mu_{11}$}};
 \node at (2.9, 7.5){\small {$\mu_{1q}$}};
 \node at (5.9,4.5){\small {$\mu_{22}$}};
 \node at (8.9,2.5){\small {$\mu_{p1}$}}; \node at (8.9,4.5){\small {$\mu_{p2}$}};
 \node at (8.9,7.5){\small {$\mu_{pq}$}};
 \end{tikzpicture}
 \caption{The contingency matrix $\mu({\bf z})$ associated to $\bf z\in\operatorname{Sym}^n({\mathbb C})$.}\label{fig:S-a-Im}
 \end{figure}

For a contingency matrix $M\in\operatorname{CM}_n$ let $X_M^{\rm cont}\subset\operatorname{Sym}^n({\mathbb C})$ be the set of ${\bf z}$ with $\mu({\bf z})=M$.

To describe the nature of $X_M^{\rm cont}$ and of its closure $\overline{X_M^{\rm cont}}$, we start with some elementary remarks. For $r\geq 0$ let $e_0,\dots, e_r$ be the standard basis of ${\mathbb R}^{r+1}$. The {\em standard $r$-simplex} $\Delta^r$ is the set
\[
\Delta^r = \operatorname{Conv}\{ e_0,\dots, e_r\} = \left\{ (x_0,\dots, x_r)\in {\mathbb R}^{r+1} \bigl| x_i\geq 0, \sum x_i=1 \right\}.
\]
 The codimension $1$ faces of $\Delta^r$ are
 \[
 \partial_i\Delta^r = \operatorname{Conv}\{ e_j,\, j\neq i\} = \Delta^r\cap \{x_i=0\}.
 \]
 Note that we have the identification
 \begin{gather}\label{eq:two-simplices}
 \Delta^r \simeq \{ 0\leq t_1\leq \cdots\leq t_r \leq 1\},\qquad x_0=t_1, x_1 = t_2-t_1, \dots, x_r=1-t_r.
 \end{gather}
 We denote
 \[
 {\overset{\circ} {\Delta}}^r = \Delta^r \setminus \bigcup_{i=0}^r \partial_i\Delta^r, \qquad \Delta^r_< = \Delta^r \setminus \partial_r\Delta^r
 \]
 the open $r$-simplex and the $m$-simplex with just the $r$th face removed. In other words, $\Delta^r_<$
 is a~cone over $\Delta^{r-1}$ but with the foundation of the cone removed.
 Note that under \eqref{eq:two-simplices}
 \begin{gather}\label{eq:cone-simplex}
 {\overset{\circ} {\Delta}}^r \simeq \{0 < t_1<\cdots < t_r < 1\},\qquad
 \Delta^r_< \simeq \{0\leq t_1\leq\cdots\leq t_r < 1\}.
 \end{gather}
 For $i=0,\dots, r-1$ we can speak about the $i$th face $\partial_i\Delta^r_<$ which is homeomorphic to
 $\Delta^{r-1}_<$.

\begin{Proposition}\label{prop;cont-cells}\quad
\begin{enumerate}\itemsep=0pt
 \item[$(a)$] Each $X_M^{\rm cont}$, $M\in\operatorname{CM}_n(p,q)$ is a cell of dimension $p+q$. More precisely,
 $X_M^{\rm cont}\simeq {\mathbb C}\times{\overset{\circ} {\Delta}}^{p-1}\times{\overset{\circ} {\Delta}}^{q-1}$.

\item[$(b)$] The closure $\overline{X_M^{\rm cont}}$ is homeomorphic to ${\mathbb C}\times\Delta^{p-1}_<\times\Delta^{q-1}_<$, and the cells lying there are given by the faces of $\Delta^{p-1}_<\times\Delta^{q-1}_<$. That is, codimension $1$ closed cells lying in $\overline{X_M^{\rm cont}}$ are
 \begin{gather*}
 \overline{ X_{\partial'_i M}^{\rm cont}} \simeq {\mathbb C}\times \partial_i \Delta^{p-1}_<\times\Delta^{q-1}_<,\qquad
 i=0, \dots, p-2,
 \\
 \overline {X_{\partial''_j M}^{\rm cont}} \simeq
 {\mathbb C}\times \Delta^{p-1}_<\times \partial_j\Delta^{q-1}_<, \qquad j=0, \dots, q-2.
 \end{gather*}
 In particular, the collection of the $X^{\rm cont}_M$ forms a quasi-regular cell decomposition of \linebreak $\operatorname{Sym}^n({\mathbb C})$ refining the stratification ${\mathcal S}_{\mathbb C}$.
 \end{enumerate}
 \end{Proposition}

We denote the collection of the $X^{\rm cont}_M$ the {\em contingency cell decomposition} of~$\operatorname{Sym}^n({\mathbb C})$ and denote ${\mathcal S}^{\rm cont}$. The $X_M^{\rm cont}$ themselves will be called the {\em contingency cells}.

\begin{proof}[Proof of Proposition \ref{prop;cont-cells}] (a) If the matrix $M=\mu({\bf z})$, i.e., the
 integers $\mu_{ij}({\bf z})$,
 are fixed, then the only data parametrizing ${\bf z}$ are the real numbers $x_1<\dots, x_p$
 and $y_1 <\cdots < y_q$. Subtracting the first elements of these sequences we get
 \[
 \{ x_1<\cdots < x_p\} \simeq {\mathbb R}\times \{ 0< x'_2 < \cdots < x'_p\}, \qquad x'_i=x_i-x_0.
 \]
 But the interval $[0,\infty)$ is identified, in a monotone way, with $[0,1)$, so
 \[
 {\mathbb R}\times \{ 0< x'_2 < \cdots < x'_p\} \simeq {\mathbb R}\times \{ 0< x'_2 < \cdots < x'_p< 1\} \simeq
 {\mathbb R}\times {\overset{\circ} {\Delta}}^{p-1},
 \]
 and similarly
 \[
 \{y_1 <\cdots < y_q\} \simeq {\mathbb R}\times{\overset{\circ} {\Delta}}^{q-1}.
 \]

 (b) The closure $\overline{X_M^{\rm cont}}$ is obtained by adding all the limit points of $X_M^{\rm cont}$. Such points
 are obtained when some of the $x_i$ or the $y_j$ merge together, and in view of
 the second identification in \eqref{eq:cone-simplex}, such mergers correspond to the faces of
 $\Delta^{p-1}_<\times\Delta^{q-1}_<$.
\end{proof}

\begin{Remark}\label{rem:stoch-realization}\quad
\begin{enumerate}\itemsep=0pt
\item[(a)] \looseness=-1 It is useful to compare the above with the concept of the geometric realization of a bi-semi\-simplicial set. That is, given a bi-semisimplicial set $Y_{\bullet,\bullet}$, its geometric realization is
 \[
|Y_{\bullet\bullet}|=\bigg( \bigsqcup_{r,s\geq 0} Y_{r,s}\times \Delta^r\times\Delta^s\bigg)\biggl/ \sim,
 \]
where $\sim$ is the equivalence relation which, for $y\in Y_{r,s}$, matches $\partial'_iy$ with $\partial_i\Delta^r\times \Delta^s$ and $\partial''_jy$ with $\Delta^r\times\partial_j\Delta^s$. This is completely analogous to the classical concept of the geometric realization of a simplicial set~\cite{gabriel-zisman, gelfand-manin}.

In our case we have an {\em augmented} bi-semisimplicial set $Y_{\bullet\bullet}$ with $Y_{r,s} = \operatorname{CM}_n(r+2, s+2)$, so the standard concept of realization is not applicable (as we cannot attach a product containing $\Delta^{-1}=\varnothing$). Instead, Proposition~\ref{prop;cont-cells} says that
 \[
 \operatorname{Sym}^n({\mathbb C}) \simeq {\mathbb C}\times \bigg( \bigsqcup_{r,s\geq -1} Y_{r,s}\times \Delta^{r+1}_<\times\Delta^{s+1}_<\bigg) \biggl/ \sim
 \]
 so we replace each $r$-simplex by the cone over it, which for $r=-1$ is taken to be just the point.

\item[(b)] Proposition \ref{prop;cont-cells} also shows that the stochastihedron ${\mathcal S}{\rm t}_n$ is simply the cell complex Poincar\'e dual to the quasi-regular cell decomposition ${\mathcal S}^{\rm cont}$ of $\operatorname{Sym}^n({\mathbb C})$. The fact that it is indeed a~cellular ball (Theorem~\ref{thm:spherical}) reflects the property that $\operatorname{Sym}^n({\mathbb C})$ is smooth (homeomorphic to a~Euclidean space). This also shows that contingency tensors of valency $d>2$ (see Remark~\ref{rem:tcont-tensors}) do not lead to a cellular complex analogous to ${\mathcal S}{\rm t}_n$, since $\operatorname{Sym}^n\big({\mathbb R}^d\big)$ is singular for $d>2$.
\end{enumerate}
\end{Remark}

{\bf Imaginary strata and Fox--Neuwirth--Fuchs cells.} We recall some constructions \linebreak from~\cite{KS-shuffle}. Consider the symmetric product $\operatorname{Sym}^n({\mathbb R})$, viewed either as the space of effective divisors ${\bf y} = \sum n_\nu\cdot y_\nu$, $y_\nu\in{\mathbb R}$, of degree $n$ or as unordered collections ${\bf y}=\{y_1,\dots, y_n\}$ possibly with repetitions. It has a quasi-regular cell decomposition into cells $K_\beta$, $\beta\in\operatorname{OP}_n$. Explicitly, if $\beta=(\beta_1,\dots,\beta_q)$, then $K_\beta$ consists of divisors $\beta_1\cdot y_1 +\cdots + \beta_q\cdot y_q$ with $y_1 < \cdots < y_q$.

Next, the imaginary part map $\operatorname{Im}\colon {\mathbb C}\to{\mathbb R}$ gives a map ${\mathfrak I}\colon \operatorname{Sym}^n({\mathbb C})\to\operatorname{Sym}^n({\mathbb R})$.
 The preimages $X_\beta^{\mathfrak I}= {\mathfrak I}^{-1}(K_\beta)$ will be called the {\em imaginary strata} of $\operatorname{Sym}^n({\mathbb C})$. They are not necessarily cells: for instance, for $\beta=(n)$ we have that $K_{(n)} = \operatorname{Sym}^n({\mathbb R})\times i{\mathbb R}$
 is the set of ${\bf y}=\{y_1,\dots, y_n\}$ with $\operatorname{Im}(y_1)=\cdots=\operatorname{Im}(y_n)$.
 In general, to say that ${\bf z} = \{z_1,\dots, z_n\}$ lies in $K_\beta$ means that there are exacty $q$ distinct values of the $\operatorname{Im}(z_\nu)$, and if we denote these values among $y_1 <\cdots < y_q$, then $y_j$ is achieved exactly $\beta_j$ times. Geometrically, we require that the $z_\nu$ lie on $q$ horizontal lines, see Fig.~\ref{fig:FNF},
 but we do not prescribe the nature of the coincidences that happen on these lines.
 \begin{figure}[h!] \centering
 \begin{tikzpicture}[scale=.4, baseline=(current bounding box.center)]

 \draw[->] (-2,0) -- (13,0);

 \draw[->] (0,-2) -- (0,11);

 \draw[dashed] (0,2) -- (11,2);

 \draw[dashed] (0,4) -- (11,4);

 \draw[dashed] (0,7) -- (11,7);

 \node at (2,2){$\bullet$};
 \node at (2,1.1){$\gamma^{(1)}_1$};

 \node at (5,2){$\bullet$};
 \node at (5,1.1){$\gamma^{(1)}_2$};

 \node at (6.5, 1) {$\cdots$};

 \node at (8,2){$\bullet$};
 \node at (8,1.1){$\gamma^{(1)}_{p_1}$};

 \node at (3,4){$\bullet$};
 \node at (3, 3.1){$\gamma^{(2)}_1$};

 \node at (5,3){$\cdots$};

 \node at (7,4){$\bullet$};
 \node at (7, 3.1){$\gamma^{(2)}_{p_2}$};

 \node at (1,7){$\bullet$};
 \node at (1, 8){$\gamma^{(q)}_1$};

 \node at (4,7){$\bullet$};
 \node at (4, 8){$\gamma^{(q)}_2$};

 \node at (9,7){$\bullet$};

 \node at (7,8){$\cdots$};

 \node at (11,7){$\bullet$};
 \node at (11, 8){$\gamma^{(q)}_{p_q}$};

 \node at (6,5.5){$\cdots\cdots$};

 \node at (-1,2){$y_1$}; \node at (-1,4){$y_2$}; \node at (-1,7){$y_q$};

 \node at (-1, 5.8){$\vdots$};


 \end{tikzpicture}
 \caption{A point $\bf z$ in a Fox--Neuwirth--Fuchs cell.}\label{fig:FNF}
 \end{figure}
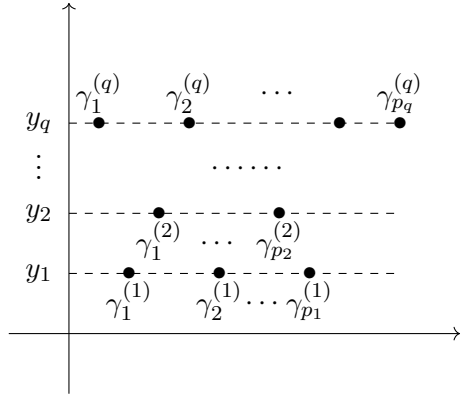

To subdivide $X_\beta^{\mathfrak I}$ further, we specify such coincidences. That is, fix a sequence of ordered partitions $\gamma = \big(\gamma^{(1)}, \dots, \gamma^{(q)}\big)$ with $\gamma^{(j)}\in\operatorname{OP}_{\beta_j}$. Let $X_{[\beta:\gamma]}$ consist of ${\bf z}\in X_\beta^{\mathfrak I}$ such that, for any $j=1,\dots, q$, the points of ${\bf z}$ lying of the $j$th horizotnal line $\operatorname{Im}^{-1}(y_j)\simeq {\mathbb R}$ belong to $K_{\gamma^{(j)}}\subset \operatorname{Sym}^{\beta_j}({\mathbb R})$. See Fig.~\ref{fig:FNF}.

 In other words, we prescribe the number of the $z_\nu$ with given imaginary parts, as well as coincidences within each value of the imaginary part. But, unlike in forming the contingency cells, we do not pay attention to possible concidences of the real parts of points with different imaginary parts. Therefore our construction is not symmetric: the imaginary part has priority over the real part.

Given $\beta=(\beta_1,\dots,\beta_q)\in\operatorname{OP}_n(q)$, a datum of a sequence $\big(\gamma^{(1)},\dots, \gamma^{(q)}\big)$, $\gamma^{(j)}\in\operatorname{OP}_{\beta_j}$, is equivalent to a single ordered partition $\gamma$ refining $\beta$, i.e., $\beta \leq\gamma$. Indeed, such a partition $\gamma$ is obtained by writing all the parts of all the $\gamma^{(j)}$ lexicographically: first the parts of $\gamma^{(1)}$, then the parts of $\gamma^{(2)}$ etc. So we will consider such pair $\beta\leq\gamma$ as a label for $X_{[\beta:\gamma]}$.

We call the $X_{[\beta:\gamma]}$ the {\em Fox--Neuwirth--Fuchs} (FNF) cells. The name ``cells'' is justified by the following fact, proved in \cite[Proposition~2.2.5]{KS-shuffle}.

\begin{Proposition}\label{prop:fuchs-complex}\quad
\begin{enumerate}\itemsep=0pt
\item[$(a)$] Each $X_{[\beta:\gamma]}$ is a cell of dimension $\ell(\beta)+\ell(\gamma)$.
\item[$(b)$] The collection of the $X_{[\beta:\gamma]}$, $\beta\leq\gamma$, forms a cell decomposition ${\mathcal S}^{\rm FNF}$ of $\operatorname{Sym}^n({\mathbb C})$ refining the complex stratification ${\mathcal S}_{\mathbb C}$. More precisely, let $\lambda\in{\mathcal P}_n$ be an unordered partition of~$n$. Then
 \[
 X_\lambda^{\mathbb C} = \bigsqcup_{\beta\leq \gamma\atop \overline\gamma=\lambda} X_{[\beta:\gamma]}.
 \]
\item[$(c)$] We have $X_{[\beta:\gamma]}\subset \overline{X_{[\beta':\gamma']}}$ if and only if $\beta\leq\beta'$ and $\gamma\leq\gamma'$ in $\operatorname{OP}_n$.
\end{enumerate}
 \end{Proposition}

\begin{Example}\quad
\begin{enumerate}\itemsep=0pt
\item[(a)] Let $n=2$ and let $\operatorname{Sym}^2_0({\mathbb C})\subset\operatorname{Sym}^2({\mathbb C})$ be the subvariety formed by $\{z_1, z_2\}$ with \mbox{$z_1+z_2=0$}. The function $\{z_1, z_2\}\mapsto w= z_1^2$ identifies $\operatorname{Sym}^2_0({\mathbb C})$ with~${\mathbb C}$. The cell decomposition ${\mathcal S}^{\rm FNF}$ induces the decomposition of this ${\mathbb C}$ into the following three cells
 \begin{gather*}
 X_{[(2): (2)]}\cap \operatorname{Sym}^2_0({\mathbb C}) = \{0\}, \qquad X_{[(2): (1,1)]} \cap\operatorname{Sym}^2_0({\mathbb C}) = {\mathbb R}_{>0}, \\ X_{[(1,1): (1,1)]} \cap \operatorname{Sym}^2_0 ({\mathbb C}) = {\mathbb C} \setminus {\mathbb R}_{\geq 0}.
 \end{gather*}
\item[(b)] The cells $X_{[\beta, (1, \dots, 1)]}$ form the cell decomposition of the open complex stratum $X_{(1,\dots, 1)}^{\mathbb C}$ used by Fox--Neuwirth~\cite{fox-neuwirth} and Fuchs~\cite{fuchs} for the study of the cohomology of the braid group $\pi_1(X_{(1,\dots, 1)}^{\mathbb C})$.
\end{enumerate}
 \end{Example}

Let $r=(r_1,\dots, r_p)\in{\mathbb Z}_{\geq 0}^p$ be a vector with non-negative (possibly zero) integer components and $\rho=\sum r_i$. We denote by $\operatorname{op}(r)\in\operatorname{OP}_\rho$ the ordered partition of~$\rho$ obtained by ``compressing''~$r$, i.e., removing the zero components. For example
 \begin{gather}\label{eq:compress}
 \operatorname{op}(2,0,1,3,0,0) = (2,1,3).
 \end{gather}
We complement Proposition \ref{prop:fuchs-complex} by

\begin{Proposition}\label{prop:cont-fuchs} The cell decomposition ${\mathcal S}^{\rm cont}$ refines ${\mathcal S}^{\rm FNF}$. More precisely, let $M=\|m_{ij}\|\in \operatorname{CM}_n(p,q)$. Then $X_M^{\rm cont}\subset X_{[\beta:\gamma]}$, where
 \begin{itemize}\itemsep=0pt
 \item $\beta = \sigma_{\rm ver}(M)$ is the vertical margin of $M$.
 \item $\gamma$, viewed as a sequence of ordered partitions $\big(\gamma^{(1)}, \cdots\gamma^{(q)}\big)$
 with $\gamma^{(\nu)} \in\operatorname{OP}_{\beta_\nu}$, has
 \[
 \gamma^{(\nu)} = \operatorname{op}(m_{\bullet,\nu}),\qquad m_{\bullet,\nu} = (m_{1,\nu},\dots, m_{p,\nu}).
 \]
 \end{itemize}
 \end{Proposition}

 The proof is obvious and left to the reader.

\section{From contingency cells to complex strata}

{\bf Four stratifications. Equivalences of contingency cells.} The stratifications of $\operatorname{Sym}^n({\mathbb C})$ that we constructed, can be represented by the following picture,
 with arrows indicating refinement:
 \begin{gather}\label{eq:4-strat}\begin{split}&
 
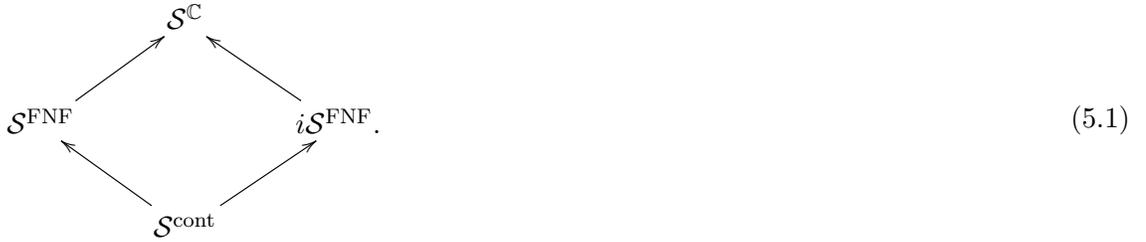
\xymatrix{
 &{\mathcal S}^{\mathbb C}&
 \\
 {\mathcal S}^{\rm FNF} \ar[ur] && i {\mathcal S}^{\rm FNF}.\ar[ul]
 \\
 & {\mathcal S}^{\rm cont}\ar[ur]\ar[ul]
 &
 }\end{split}
 \end{gather}
Here $i{\mathcal S}^{\rm FNF}$ is the ``dual Fox--Neuwirth--Fuchs'' cell decomposition, obtained from ${\mathcal S}^{\rm FNF}$ by applying either of the two the automomorphism of $\operatorname{Sym}^n({\mathbb C})$ (they give the same stratification up to relabeling):
 \begin{itemize}\itemsep=0pt
 \item The holomorphic automorphism induced by $i\colon {\mathbb C}\to{\mathbb C}$ (multiplication by $i$).
 \item The non-holomorphic automorphism induced by $\sigma\colon {\mathbb C}\to{\mathbb C}$, $x+iy\mapsto y+ix$.
 \end{itemize}

\begin{Remark}\label{rem;:oot-arr} Any real hyperplane arrangement ${\mathcal H}\subset{\mathbb R}^n$ gives three stratifications~${\mathcal S}^{(0)}$,~${\mathcal S}^{(1)}$ and~${\mathcal S}^{(2)}$ of ${\mathbb C}^n$, see \cite[Section~2]{KS-hyp-arr}. For example, ${\mathcal S}^{(0)}$ consists of generic parts of the complex flats of ${\mathcal H}$ and ${\mathcal S}^{(2)}$ consists of ``product cells'' $C+iD$ where $C$, $D$ are faces of~${\mathcal H}$. Taking for ${\mathcal H}$ the {\em root arrangement} in ${\mathbb R}^n$, i.e., the system of hyperplanes $\{x_i=x_j\}$, we obtain our stratifications~${\mathcal S}_{\mathbb C}$,~${\mathcal S}^{\rm FNF}$ and~${\mathcal S}^{\rm cont}$ as the images of ${\mathcal S}^{(0)}$, ${\mathcal S}^{(1)}$ and ${\mathcal S}^{(2)}$ under the projection $\pi$ of~\eqref{eq:pi}.
 \end{Remark}

We are interested in the way the complex strata (from ${\mathcal S}_{\mathbb C}$) are assembled out of the cells of~${\mathcal S}^{\rm cont}$. Recall that the partial order $\leq$ on $\operatorname{CM}_n$ is the ``envelope'' of two partial orders~$\leq'$ and~$\leq''$ given by the horizontal and vertical contractions $\partial'_i$, $\partial''_j$, so that $\partial'_iM\leq' M$ and $\partial''_j M\leq'' M$.

\begin{Definition} We say that an inclusion $N\leq M$ in $\operatorname{CM}_n$ is an {\em equivalence}, if $X^{\rm cont}_N$ and $X^{\rm cont}_M$ lie in the same complex stratum. Similarly, we say that $N\leq' M$, resp.~$N\leq'' M$ is a~{\em horizontal equivalence}, resp.\ a~{\em vertical equivalence}, if $X^{\rm cont}_N$ and $X^{\rm cont}_M$ lie in the same complex stratum.
\end{Definition}

It is enough to describe ``elementary'' horizontal and vertical equvialences. That is, we call the contraction $\partial'_i$ {\em anodyne} for $M$, if $\partial'_iM\leq' M$ is a horizontal equvialence. Similarly, the vertical contraction $\partial''_j$ is called anodyne for $M$, if $\partial''_jM\leq'' M$ is a vertical equivalence. Thus arbitrary horizontal (resp.\ vertical) equivalences are given by chains of anodyne horizontal (resp.\ vertical) contractions.

Given two integer vectors $r=(r_1,\dots, r_q), s=(s_1,\dots, s_q)\in{\mathbb Z}_{\geq 0}^q$, we say that they are {\em disjoint}, if $r_js_j=0$ for each $j=1,\dots, q$, i.e., in each position at least one of the components of~$r$ and~$s$ is zero.

\begin{Proposition}Let $M\in \operatorname{CM}_n(p,q)$ be given.
\begin{enumerate}\itemsep=0pt
\item[$(a)$] $\partial'_i$ is anodyne for $M$ if and only if the $(i+1)$st and the $(i+2)$nd columns
 of $M$ $($that are added together under $\partial'_i)$ are disjoint.
\item[$(b)$] Similarly, $\partial''_j$ is anodyne for $M$, if and only if the $(j+1)$ st and $(j+2)$nd rows of $M$ are disjoint.
\end{enumerate}
 \end{Proposition}

\begin{proof}This is clear, as, say, columns being disjoint means precisely that the multiplicities (considered as an unordered collection) do not change after adding the columns.
\end{proof}

{\bf The upper and lower bound of ${\mathcal S}^{\rm FNF}$ and $i{\mathcal S}^{\rm FNF}$.} The relation between the four stratifications in~\eqref{eq:4-strat} can be expressed as follows.

 \begin{Theorem}\label{prop:upper-bound}\quad
\begin{enumerate}\itemsep=0pt
\item[$(a)$] We have
 \[
{\mathcal S}^{\rm FNF} \wedge i{\mathcal S}^{\rm FNF} = {\mathcal S}^{\rm cont}.
 \]
More precisely, ${\mathcal S}^{\rm cont}$ is the coarsest stratification with connected strata that refines both ${\mathcal S}^{\rm FNF}$ and~$i{\mathcal S}^{\rm FNF}$.
\item[$(b)$] We also have
 \[
{\mathcal S}^{\rm FNF} \vee i{\mathcal S}^{\rm FNF} = {\mathcal S}_{\mathbb C}.
 \]
More precisely, ${\mathcal S}_{\mathbb C}$ is the finest stratification of which both ${\mathcal S}^{\rm FNF}$ and $i{\mathcal S}^{\rm FNF}$ are refinements.
\end{enumerate}
 \end{Theorem}

\begin{proof} We first prove part (b) of the theorem. Let~$W$, resp.~$W'$, resp.~$W''\subset \operatorname{CM}_n\times\operatorname{CM}_n$ be the set of pairs $(N,M)$ such that $N\leq M$ and the inclusion is an equivalence, resp.~$N\leq' M$ and the inclusion is a horizontal equivalence, resp.~$N\leq'' M$ and the inclusion is a vertical equivalence. Let $R$, $R'$, $R''$ be the equivalence relations generated by $W$, $W'$, $W''$. Since the strata of ${\mathcal S}_{\mathbb C}$ are connected, we have, first of all:

\begin{Lemma} $N\sim_R M$ if and only if $X^{\rm cont}_N$ and $X^{\rm cont}_M$ lie in the same complex stratum.
\end{Lemma}

 Theorem \ref{prop:upper-bound}(b) will follow from this lemma and the next statement.

\begin{Proposition}\label{prop:R'}\quad
\begin{enumerate}\itemsep=0pt
\item[$(a)$] $N\sim_{R'}M$ if and only if $X^{\rm cont}_N$ and $X^{\rm cont}_M$ lie in the same cell of ${\mathcal S}^{\rm FNF}$.
\item[$(b)$] $N\sim_{R''}M$ if and only if $X^{\rm cont}_N$ and $X^{\rm cont}_M$ lie in the same cell of $i{\mathcal S}^{\rm FNF}$.
\end{enumerate}
 \end{Proposition}

\begin{proof} It is enough to show (a), since (b) is similar. We first prove the ``only if'' part, that is, whenever $\partial'_i$ is anodyne for $M$, the cells $X^{\rm cont}_{\partial'_i M}$ and $X^{\rm cont}_M$ lie in the same Fox--Neuwirth--Fuchs cell. But this is obvious from comparing Figs.~\ref{fig:S-a-Im} and~\ref{fig:FNF}: if the $(i+1)$st and $(i+2)$nd columns of $M$ are disjoint, then the multiplicity structure on each horizontal line is unchanged after a generation resulting in adding these columns.

Let us now prove the ``if'' part. Since each FNF cell is connected (being a cell), it suffices to prove the following: whenever $X_{\partial'_iM}^{\rm cont}$ and $X^{\rm cont}_M$ lie in the same FNF cell, the contraction $\partial'_i$ is anodyne for~$M$. But this is again obvious, since a non-anodyne contraction will change the multiplicity structure on some horizontal line. Proposition~\ref{prop:R'} is proved.
\end{proof}

This also completes the proof of Proposition \ref{prop:upper-bound}(b).

We now prove Proposition~\ref{prop:upper-bound}(a). Let $M\in\operatorname{CM}_n(p,q)$. By Proposition~\ref{prop:cont-fuchs},
 \[
 X_M^{\rm cont} \subset X_{[\beta:\gamma]}\cap iX_{[\alpha:\delta]},
 \]
where $\alpha=\sigma_{\rm hor}(M)$ and $\beta=\sigma_{\rm ver}(M)$ are the margins of $M$ and $\gamma$, resp.~$\delta$ is obtained by compressing, cf.~\eqref{eq:compress}, the rows, resp. columns of $M$. In particular, the size $p\times q$ of $M$ is determined as $p=\ell(\alpha)$, $q=\ell(\beta)$ from the unique cells $X_{[\beta:\gamma]}$ and $iX_{[\alpha:\delta]}$ containing $X_M^{\rm cont}$. Note that $\dim X_M^{\rm cont} = p+q$. This means the following: given any two cells $X_{[\beta:\gamma]}\in {\mathcal S}^{\rm FNF}$ and $iX_{[\alpha:\delta]}\in i{\mathcal S}^{\rm FNF}$, all contingency cells contained in their intersection, have the same dimension. Since, the union of such cells is the intersection $X_{[\beta:\gamma]}\cap iX_{[\alpha:\delta]}$, we conclude that by taking the connected components of all the $X_{[\beta:\gamma]}\cap iX_{[\alpha:\delta]}$, we get precisely all the contingency cells.
\end{proof}

{\bf Corollaries for constructible sheaves.} Fix a base field ${\bf k}$. For a stratified space $(X,{\mathcal S})$ we denote by $\operatorname{Sh}(X,{\mathcal S})$ the category formed by sheaves~${\mathcal F}$ of ${\bf k}$-vector spaces which are constructible with respect to~${\mathcal S}$, i.e., such that restriction of ${\mathcal F}$ to each stratum is locally constant. The following is standard, see, e.g., \cite[Proposition~1.]{KS-hyp-arr}.

 \begin{Proposition} Suppose that $(X,{\mathcal S})$ be a quasi-regular cellular space with the poset $({\mathcal C},\leq)$ of cells. Then $\operatorname{Sh}(X,{\mathcal S})$ is identified with $\operatorname{Rep}({\mathcal C})$, the category of representations of $({\mathcal C},\leq)$ in ${\bf k}$-vector spaces.
 \end{Proposition}

 We recall that a representation of $({\mathcal C},\leq)$ is a datum, consisting of:
 \begin{itemize}\itemsep=0pt
 \item[(0)] ${\bf k}$-vector spaces $F_\sigma$, given for any $\sigma\in {\mathcal C}$.
 \item[(1)] Linear maps $\gamma_{\sigma, \sigma'}\colon F_\sigma\to F_{\sigma'}$ given for any $\sigma\leq\sigma'$ and satisfying
 \item[(2)] $\gamma_{\sigma,\sigma}=\operatorname{Id}\nolimits$, and $\gamma_{\sigma,\sigma''} = \gamma_{\sigma',\sigma''}\circ
 \gamma_{\sigma,\sigma'}$ for any $\sigma\leq\sigma'\leq\sigma''$.
 \end{itemize}

For ${\mathcal F}\in\operatorname{Sh}(X,{\mathcal S})$, the corresponding representation has $F_\sigma = \Gamma(\sigma, {\mathcal F}|_\sigma)$, the space of sections of ${\mathcal F}$ on $\sigma$ (or, what is canonically the same, the stalk at any point of~$\sigma$). The map $\gamma_{\sigma,\sigma'}$ is the {\em generalization map} of~${\mathcal F}$, see \cite[Section~1D]{KS-hyp-arr} and references therein.

\begin{Corollary}\quad
\begin{enumerate}\itemsep=0pt
\item[$(a)$] The category $\operatorname{Sh}\big(\operatorname{Sym}^n({\mathbb C}), {\mathcal S}^{\rm cont}\big)$ is equivalent to $\operatorname{Rep}(\operatorname{CM}_n, \leq)$.
\item[$(b)$] The category $\operatorname{Sh}\big(\operatorname{Sym}^n({\mathbb C}), {\mathcal S}_{\mathbb C}\big)$ is equivalent to the full subcategory formed by representations $(F_M,( \gamma_{N,M})_{N\leq M})$, such that $\gamma_{\partial'_iM, M}$, resp.~$\gamma_{\partial''_jM, M}$ is an isomorphism whenever the contraction $\partial'_i$, resp. $\partial''_j$ is anodyne.
\end{enumerate}
 \end{Corollary}

 Proposition \ref{prop:upper-bound} implies the following.

 \begin{Corollary} An ${\mathcal S}^{\rm cont}$-constructible sheaf on $\operatorname{Sym}^n({\mathbb C})$ is ${\mathcal S}_{\mathbb C}$-constructible, if and only if it is constructible for both ${\mathcal S}^{\rm FNF}$ and $i{\mathcal S}^{\rm FNF}$.
 \end{Corollary}

\newpage

\appendix\renewcommand{\thefootnote}{}

\section[Counting contingency matrices. Appendix by Pavel Etingof]{Counting contingency matrices.\\ Appendix by Pavel Etingof\footnote{Department of Mathematics, MIT, Cambridge MA 02139, USA}
\footnote{E-mail: \href{mailto:etingof@math.mit.edu}{etingof@math.mit.edu}}}

\begin{Definition} A {\em generalized contingency matrix} is a rectangular matrix~$M$ whose entries $m_{ij}$ are nonnegative integers. The weight of a generalized contingency matrix is $\sum m_{ij}$.
\end{Definition}

Thus, a contingency matrix is a generalized contingency matrix without zero rows or columns. The following is obvious.

\begin{Lemma}The number of generalized contingency matrices of size $p\times q$ and of weight $n$ is $\binom{n+pq-1}{n}$.
\end{Lemma}

Let ${\mathfrak m}_{pq}(n)$ be the number of contingency matrices of this size and weight and $\mathfrak{M}(n)=\|{\mathfrak m}_{pq}(n)\|_{p,q=1, \dots, n}$.

\begin{Lemma}\label{l1} We have
\[
\sum_{i\le p; j \le q}\binom{p}{i}\binom{q}{j} {\mathfrak m}_{ij}(n)=\binom{n+pq-1}{n}.
\]
\end{Lemma}

\begin{proof} Every generalized contingency matrix $M$ of weight $n$ defines subsets $S\subset [1,p]$, $T\subset [1,q]$ (corresponding to zero rows and zero columns of~$A$) and a contingency matrix $M_+$ of size $(p-|S|)\times (q-|T|)$ and weight $n$ obtained by deleting the zero rows and columns from~$A$. Clearly, the assignment $M\mapsto (S,T,M_+)$ is a bijection. This implies the statement.
\end{proof}

Let $P(n)$ be the unipotent lower triangular matrix such that $P(n)_{pi}=\binom{p}{i}$. The following corollary of Lemma~\ref{l1} is immediate.

\begin{Corollary}\label{l2} We have $P(n)\mathfrak{M} (n)P(n)^t=B(n)$, where $B(n)_{pq}=\binom{n+pq-1}{n}$. Thus $\mathfrak{M} (n)=P(n)^{-1}B(n)\big(P(n)^{-1}\big)^t$. In particular, $\det\mathfrak{M}(n)=\det B(n)$.
\end{Corollary}

Note also that
 \[
 \big(P(n)^{-1}\big)_{pi}=(-1)^{p-i}\binom{p}{i}.
 \]
 Indeed, denote the matrix in the r.h.s.\ by $P_*(n)$. Then
\[
 (P (n)P_*(n))_{pj}=\sum_{p\ge i\ge j} (-1)^{i-j}\binom{p}{i}\binom{i}{j}=\binom{p}{j}\sum_{p\ge i\ge j}(-1)^{i-j}\binom{p-j}{i-j}=\delta_{pj}.
\]

Thus we get

\begin{Corollary}\label{l3}
\[
{\mathfrak m}_{pq}(n)=\sum_{i,j}(-1)^{i+j+p+q}\binom{p}{i}\binom{q}{j}\binom{n+ij-1}{n}.
\]
\end{Corollary}

Recall \cite{stanley} that the {\em $($unsigned$)$ Stirling numbers of the first kind} $c(n,k)$ are defined by the generating function
 \[
x(x+1)(x+2)\cdots (x+n-1) = \sum_{k=1}^{n} c(n,k) x^k.
\]

\begin{Proposition}\label{l4} We have
\[
B(n)=\frac{1}{n!}V(n)\cdot \operatorname{diag}(c(n,1),\dots ,c(n,n))\cdot V(n)^t,
\]
where $V(n)$ is the $($modified$)$ Vandermonde matrix, $V(n)_{ik}=i^k$. Hence
\[
\mathfrak{M}(n)=\frac{1}{n!}Q(n)\cdot \operatorname{diag}(c(n,1),\dots ,c(n,n))\cdot Q(n)^t,
\]
where $Q(n):=P(n)^{-1}V(n)$.
\end{Proposition}

\begin{proof} We have
\[
B(n)_{pq}=\frac{1}{n!}\sum_k c(n,k) p^k q^k,
\]
which implies the first statement. The second statement follows from the first one and
Corollary~\ref{l2}.
\end{proof}

\begin{Proposition}\label{l5} We have
\[
d_n:=\det \mathfrak{M}(n)=\frac{n!\prod\limits_{i=1}^{n-1}c(n,i)}{\prod\limits_{i=1}^{n-1}\binom{n}{i}}.
\]
In particular, the fraction in the r.h.s.\ is an integer.
\end{Proposition}

\begin{proof}We have $\det V(n)= n!\prod\limits_{1\le j<i\le n}(i-j)$ (the Vandermonde determinant). This, together with Proposition \ref{l4}, implies the statement after simplifications (using that $c(n,n)=1$).
\end{proof}

\begin{Example} We have $d_1=1$, $d_2=4$, $d_4=99$.
\end{Example}

Thus by summing over $p$, $q$ we get
\begin{Corollary}\label{l6}
\[
\sum_{p,q=1}^n {\mathfrak m}_{pq}(n)=\sum_{i,j}(-1)^{i+j}\binom{n+1}{i+1}\binom{n+1}{j+1}\binom{n+ij-1}{n}.
\]
\end{Corollary}

\begin{Proposition}\label{l7} The matrix $Q(n)$ is upper triangular, and its entries are $p! S(k,p)$, where $S(k,p)$ are the Stirling numbers of the second kind~{\rm \cite{stanley}}. In particular, the diagonal entries of~$Q(n)$ are~$k!$.
\end{Proposition}

\begin{proof} We have $Q(n)=P(n)^{-1}V(n)$. Thus
\[
Q(n)_{pk}=\sum (-1)^{p-i}\binom{p}{i}i^k=S(k,p)p!,
\]
the last equality being the definition of $S(k,p)$. It is well known that $S(k,p)=0$ if $p>k$, which implies the statement.
\end{proof}

\begin{Corollary}\label{l8}The Gauss decomposition of $V(n)$ is given by
\[
V(n)=P(n)\operatorname{diag}(1!,2!,\dots ,n!) S(n),
\]
where $S(n)$ is the unipotent upper triangular matrix whose entries are $S(n)_{pk}=S(k,p)$ for $k,p\le n$.
\end{Corollary}

\begin{proof} This follows from Proposition~\ref{l7} since $Q(n)=\operatorname{diag}(1!,2!,\dots ,n!)S(n)$.
\end{proof}

\begin{Corollary}\label{l9}The $($opposite$)$ Gauss decomposition of $\mathfrak{M}(n)$ is
\[
\mathfrak{M}(n)=\frac{1}{n!}S_*(n)\cdot \operatorname{diag}\big((1!)^2c(n,1),\dots, (n!)^2c(n,n)\big)\cdot S_*(n)^t,
\]
where $S_*(n)_{pk}:=p! S(k,p)/k!$.
\end{Corollary}

\begin{proof} This follows from Proposition \ref{l4}.
\end{proof}

\begin{Corollary}\label{l10} The matrix $\mathfrak{M}(n)$ is totally positive, i.e., all of its determinants of all sizes are positive.
\end{Corollary}

\begin{proof} Let $G={\rm GL}_n({\mathbb R})$. Let $U^+,U^-\subset G$ be the subgroups of unipotent upper and lower triangular matrices, and $T$ be the torus of diagonal matrices. Let also $G_{>0}\subset G$ be the set of totally positive matrices. For distinct $i,j\in \{1,\dots, n\}$ and $a\in{\mathbb R}$ let $e_{ij}(a)$ be the elementary matrix which has $1$'s on the diagonal, $a$ in the position $(i,j)$ and $0$ elsewhere.
Recall \cite{loewner, lusztig} that
\[
G_{>0} = U^+_{>0}T_{>0}U^-_{>0} = U^-_{>0}T_{>0}U^+_{>0},
\]
where
\begin{itemize}\itemsep=0pt
\item $T_{>0}\subset T$ is the subset of diagonal matrices with all the diagonal entries positive.
\item $U^+_{>0}\subset U^+$ is the subset of matrices of the form $\prod\limits_{i<j} e_{ij}(a_{ij})$ where all $a_{ij}>0$ and the product is taken in the order of a reduced decomposition of the maximal element in $S_n$. Alternatively. $U^+_{>0}$ can be defined as the interior of the closed subset in $U^+$ formed by matrices with all minors non-negative.
\item $U^-_{>0}$ is defined similarly using $e_{ij}(a_{ij})$ with $i>j$ and $a_{ij}>0$ or, equivalently, as
the interior of the subset in $U^-$ formed by matrices with all minors non-negative.
\end{itemize}
It is well known \cite{gantmacher} that the matrix $V(n)$ is totally positive (it follows from the fact that the Schur polynomials have positive coefficients). Thus it follows from Corollary \ref{l8} that $S(n)$ is totally positive. But then by Corollary~\ref{l9} we get that $\mathfrak{M}(n)$ is totally positive.
\end{proof}

We also obtain
\begin{Corollary}\label{l11}We have
\[
\sum_{p,q}{\mathfrak m}_{pq}(n)=\frac{1}{n!}\sum_{p,q,k}c(n,k) p! S(k,p) q! S(k,q).
\]
\end{Corollary}

Since $\sum_p p!S(k,p)=F(k)$, the Fubini numbers ($=$ ordered Bell numbers~\cite{stanley}), we get

\begin{Corollary}\label{l11}
We have
\[
{\mathfrak m} (n):=\sum_{p,q}{\mathfrak m}_{pq}(n)=\frac{1}{n!}\sum_{k}c(n,k)F(k)^2.
\]
\end{Corollary}

\begin{Example} The Fubini numbers are $1,3,13,\dots$, and $c(3,i)$ are $2,3,1,0,\dots $, so
${\mathfrak m}(3)=\big(2\cdot 1^2+3\cdot 3^2+1\cdot 13^2\big)/6=33$. This is the total number of faces in the stochastihedron ${\mathcal S}{\rm t}_3$, see Example~\ref{ex:St-3}.
\end{Example}

\subsection*{Acknowledgements}

We are happy to dedicate this paper to Dmitry Borisovich Fuchs. Among several wonderful things he has done in mathematics, he is one of the pioneers in the study of cellular decompositions for symmetric products.

We are grateful to Pavel Etingof for valuable discussions and for agreeing to include his work as an appendix to our paper. V.S.\ is grateful to the organizers of a conference in Z\"urich in August 2019 where he had a chance to meet~P.E. We would like to thank Sergei Fomin for pointing out the relevance of the paper~\cite{petersen} to our work and for pointing out several misprints in the earlier version. We are also grateful to the referees for their useful remarks and corrections. The research of M.K.~was supported by World Premier International Research Center Initiative (WPI Initiative), MEXT, Japan.

\pdfbookmark[1]{References}{ref}
\LastPageEnding

\end{document}